\documentclass[preprint,11pt]{elsarticle}
\usepackage{amsfonts}
\usepackage{mathrsfs}
\usepackage{tipa}
\usepackage{amssymb}
\usepackage{latexsym,bm}
\usepackage{amsthm}
\usepackage{yhmath}
\usepackage{booktabs}
\usepackage{multirow}
\usepackage[center]{caption}
\usepackage{graphics}
\usepackage{epsfig}
\usepackage{epstopdf}
\usepackage{subfig}
\usepackage{color}
\large\normalsize
\usepackage{amssymb}
\newtheorem{theorem}{Theorem}[section]
\newtheorem{lemma}[theorem]{Lemma}

\journal{arXiv}

\begin{document}
\newcommand{\Q}{\mathbb{Q}}
\newcommand{\N}{\mathbb{N}}
\newcommand{\Z}{\mathbb{Z}}
\newcommand{\R}{\mathbb{R}}
\newcommand{\M}{\mathbb{M}}
\newcommand{\om}{\omega}
\newcommand{\tx}[1]{\quad\mbox{#1}\quad}

\begin{frontmatter}
\title{{\bf{Novel numerical analysis for simulating the generalized 2D multi-term time fractional Oldroyd-B fluid model }}\tnoteref{label1}}
%\tnotetext[label1]{}

\author{Yanqin Liu\fnref{label2,label3}}
\ead{yqliumath@163.com}

\author{Fawang Liu\corref{cor1}\fnref{label3}}
\ead{f.liu@qut.edu.au}

\author{Libo Feng\fnref{label3}\corref{}}
\ead{libo.feng@hdr.qut.edu.au}

\author{Baogui Xin\fnref{label4}\corref{}}
\ead{xin@tju.edu.cn}

\cortext[cor1]{Corresponding author.}

\address[label2]{School of Mathematical Sciences, Dezhou University, Dezhou. 253023, China}
\address[label3]{School of Mathematical Sciences, Queensland University of Technology, GPO Box 2434, Brisbane, Qld. 4001, Australia}
\address[label4]{Nonlinear Science Center, College of Economics and Management, Shandong University and Technology, Qingdao 266590, China}

\begin{abstract}
In this paper, we consider the finite difference method for the generalized two-dimensional (2D) multi-term time-fractional Oldroyd-B fluid model, which is a subclass of non-Newtonian fluids. Different from the general multi-term time fractional equations, the generalized fluid equation not only has a multi-term time derivative but also possess a special time fractional operator on the spatial derivative.
Firstly, a new discretization of the time fractional derivative is given. And a vital lemma, which plays an important role in the proof of
stability, is firstly proposed. Then the new finite difference scheme is constructed. Next, the unique solvability, unconditional stability and
convergence of the proposed scheme are proved by the energy method. Numerical examples are given to verify the numerical accuracy and
efficiency of the numerical scheme as compared to theoretical analysis, and this numerical method can extended to solve other non-Newtonian fluid models.
\end{abstract}

\end{frontmatter}

\section{Introduction}
\label{sec:1}
\par In the last few decades, non-Newtonian fluids which do not satisfy a linear relationship between the stress tensor and the deformation
tensor have been widely applied in engineering and industry. The constitutive equation of non-Newtonian fluids is much more complex than its
Newtonian counterparts, and the constitutive equations involving fractional calculus have been proved to be a valuable tool to handle
viscoelastic properties \cite{Bagley,Friedrich} and some results are obtained which are in good agreement with the experimental data
\cite{Makris,Jimenez}.
\par One particular subclass of non-Newtonian fluids is the generalized Oldroyd-B fluid, which has been found to approximate the response to
many dilute polymeric liquids. Some recent work regarding the generalized Oldroyd-B fluids can be found in references
\cite{Qi,Zhao6,Zhao,Jiang}. The fundamental electromagnetic relations have been summarized by Sutton \cite{Sutton}. For the
magnetohydrodynamic(MHD) flow, Zheng et al. \cite{Liu,Zheng} discussed the flow between two plates with slip boundary conditions and obtain
the exact solution in terms of Fox $H$-function by some transform techniques. Khan et al. \cite{Khan1,Khan2} investigated the MHD flow of a
generalized Oldroyd-B fluid in a circular pipe and a porous space, respectively.
\par One important part is the following incompressible Oldroyd-B fluid which is bounded by two infinite rigid plates, when a magnetic field
is imposed on the above flow under the assumption of low magnetic Reynolds number. Fetecau et al. \cite{Fetecau} considered the two
dimensional case:
$$(1+\lambda D^\alpha_t)\frac{\partial u(x,y,t)}{\partial t}=\upsilon(1+\theta D^\beta_t)(\frac{\partial^2}{\partial
x^2}+\frac{\partial^2}{\partial y^2})u(x,y,t),$$
where $\lambda$ and $\theta$ are relaxation and retardation times, $\nu=\frac{\mu}{\rho}$ is the kinematic viscosity of the fluid, $\rho$ is
the density of the fluid, $\mu$ is the dynamic viscosity coefficient of the fluid. $D^\alpha_t$, $D^\beta_t$ are the time fractional
operators, $0<\alpha,\beta<1$ and $u(x,y,t)$ is the velocity.
Khan et al. \cite{Khan2} considered the following generalized Oldroyd-B fluid in a porous medium with the influence of Hall current
\begin{align*}
   (1+\lambda{D}_t^{\alpha})\frac{\partial u(x,t)}{\partial t}&=\nu(1+\theta{D}_t^\beta)\frac{\partial^2u(x,t)}{\partial
   x^2}-\frac{\nu\varphi_1}{k}(1+\theta{D}_t^{\beta})u(x,t)\\
   &-\frac{\sigma B_0^2}{\rho(1-i\phi)}(1+\lambda{D}_t^{\alpha})u(x,t),
\end{align*}
where $k$ is the permeability of the porous medium, $\varphi_1$ is the porosity of the medium, $\phi$ is the Hall parameter, $B_0$ is the
magnetic intensity and $\sigma$ is the electrical conductivity.
\par Stimulated by the above research in this field, we will give the following two dimensional Oldroyd-B fluid with the influence of Hall
current
\begin{align*}
   (1+\lambda{D}_t^{\alpha})\frac{\partial u(x,y,t)}{\partial t}&=\nu(1+\theta{D}_t^\beta)(\frac{\partial^2u}{\partial
   x^2}+\frac{\partial^2u}{\partial y^2})-\frac{\nu\varphi_1}{k}(1+\theta{D}_t^{\beta})u(x,y,t)\\\nonumber
   &-\frac{\sigma B_0^2}{\rho(1-i\phi)}(1+\lambda{D}_t^{\alpha})u(x,y,t),
\end{align*}
we will give the detailed derivation of this fluid model in Section 2.
\par The research of the incompressible Oldroyd-B fluid is limited to the analytical solution, numerical methods with supporting stability
and convergence analysis are limited. In order to give the numerical methods and the discrete scheme, we will consider the following
generalized two-dimensional multi-term time fractional non-Newtonian diffusion equation:
\begin{align}
& \sum_{l=1}^pa_lD^{\gamma_l}_tu+b_1\frac{\partial u}{\partial t}+\sum_{m=1}^qc_mD^{\alpha_m}_tu+b_2u \nonumber\\
= ~~& b_3(\frac{\partial^2 u}{\partial x^2}+\frac{\partial^2 u}{\partial y^2})
+\sum_{r=1}^sd_rD^{\beta_r}_t(\frac{\partial^2 u}{\partial x^2}+\frac{\partial^2 u}{\partial y^2})
+f(x,y,t), \label{e1}
\end{align}
where $u=u(x,y,t),$ $(x,y)\in \Omega=(0,L_x)\times(0,L_y),0<t\leq T,$ and with the following initial condition
$$u(x,y,0)=\varphi(x,y), u_t(x,y,0)=\phi(x,y),(x,y)\in \Omega$$
and the boundary conditions
$$u(0,y,t)=u(L_x,y,t)=0, 0\leq y\leq L_y, t>0,$$
$$u(x,0,t)=u(x,L_y,t)=0, 0\leq x\leq L_x, t>0,$$
where $a_l>0,b_i>0,c_m>0,d_r>0,$ $1<\gamma_l<2, 0<\alpha_m,\beta_r<1$, $l=1,\cdots, p, i=1,2,3,m=1,\cdots,q, r=1,\cdots,s$, and the Caputo
time frational derivative $D^\gamma_t u$, $D^\alpha_t u$ are given by \cite{Hilfer,Podlubny}
$$D^\gamma_t u=\frac{1}{\Gamma(2-\gamma)}\int_0^t(t-s)^{1-\gamma}\frac{\partial^2 u(x,y,s)}{\partial s^2}ds, 1<\gamma<2,$$
$$D^\alpha_t u=\frac{1}{\Gamma(1-\alpha)}\int_0^t(t-s)^{-\alpha}\frac{\partial u(x,y,s)}{\partial s}ds, 0<\alpha<1,$$
where $\Gamma(\cdot)$ is the Gamma function.
\par Although there are some literatures \cite{Liu,Zheng,Ming,Jiang1} give the analytical solutions of the generalized Oldroyd-B fluid, but
they are always given in series form with generalized G or $H$-function. Therefore, numerical method is a promising tool to solve these
equations. And up to now, numerical methods to solve fractional diffusion equation mainly are finite difference methods
\cite{Liu04,Liu07,Zhuang08,Zhuang09,Liu13,Liu15,Feng2}, finite element method \cite{Zhao16,Yang17,ZhaoX17,Fan18}, finite volume methods
\cite{Liu14,Li17}, spectral methods \cite{Zeng14,Zheng15,Zhaox} and meshless methods \cite{Liuq14,Liuq15}. Bazhlekova et al.
\cite{Bazhlekova} proposed a finite difference method to solve the viscoelastic flow with generalized fractional Oldroyd-B model, fractional
operator is Riemann-Liouville time fractional derivative and they utilised the Gr\"{u}nwald-Letnikov formula to approximate it, which was low
accuracy and lacked theoretical analysis. Recently, Feng et al. \cite{Fengf} gave the numerical solution to these problems, but it confined to
one-dimensional case, the two-dimensional case is seldom solved, and the temporal convergence order we get in this paper is
${\min\{3-\gamma_l,2-\alpha_m,2-\beta_r\}}$, this is also better than \cite{Fengf} which the temporal convergence is only first order.
\par The outline of the paper is as follows. In section 2, preliminary knowledge is given, in which a new numerical scheme to discretise the
time fractional derivative is proposed. In section 3, we develop the finite difference method for the generalized Oldroyd-B fluid model and
derive the implicit scheme. And we proceed with the proof of the stability and convergence of the scheme by energy method and discuss the
solvability of the numerical scheme in section 4. In section 5, we present two numerical examples to demonstrate the effectiveness of our
method and some conclusions are drawn finally.

\section{Formulation of the Multi-term time fractional flow model}

\par We impose a magnetic field in the positive $y$-axis with intensity $B_0$ and the electrical conductivity is $\sigma$. Suppose that the
main flow only takes place along the $x-$axis, then we shall assume the velocity field and the extra stress of the form
\begin{equation} \label{e2}
\textbf{V}=\textbf{V}(x,y,t)=u(x,y,t)\textbf{i},~\textbf{S}=\textbf{S}(x,y,t)=S(x,y,t)\textbf{i}
\end{equation}
where \textbf{i} is the unit vector in the $z-$direction of the Cartesian coordinate system $x, y$ and $z$.
\par The conservation equation of an incompressible fluid is
\begin{equation} \label{e3}
div \textbf{V}=0,
\end{equation}
\begin{equation} \label{e4}
\rho \frac{d\textbf{V}}{dt}=-\nabla P+div \textbf{S}+\textbf{J}\times\textbf{B}+\textbf{r},
\end{equation}
where $\rho$ is the fluid density, $P$ is the hydrostatic pressure, $\nabla$ is the gradient operator, $\textbf{r}$ is the Darcy resistance
for an Oldroyd-B fluid, and $\textbf{J}$ is the current density, $\textbf{B}$ is the total magnetic field so that
$\textbf{B}=\textbf{B}_0+\textbf{b}$, $\textbf{B}_0 $ and $\textbf{b} $ are the applied and induced magnetic fields, respectively. By Eq.(2),
the continuous equation (3) holds automatically. The constitutive equation for a generalized Oldroyd-B fluid is defined as\cite{Hilfer}:
 \begin{equation} \label{e5}
(1+\lambda\frac{D^\alpha}{Dt^\alpha})\textbf{S}=\mu(1+\theta\frac{D^\beta}{Dt^\beta})\textbf{A},\quad (0<\alpha,\beta<1)
\end{equation}
where $\mu$ is the dynamic viscosity and $\lambda$ and $\theta$ are relaxation and retardation times, and $\alpha$ and $\beta$ are fractional
calculus orders, and $\textbf{A}=\textbf{L}+\textbf{L}^T (\textbf{L}=\nabla\textbf{V})$ denotes the first Rivlin-Ericksen tensor. The material
derivative operators $\frac{D^\alpha}{Dt^\alpha}$ and $\frac{D^\beta}{Dt^\beta}$ can be expressed as
 \begin{equation} \label{e6}
\frac{D^\alpha}{Dt^\alpha}\textbf{S}=D^\alpha_t\textbf{S}+(\textbf{V}\cdot\nabla)\textbf{S}-\textbf{L}\textbf{S}-\textbf{S}\textbf{L}^T,
\end{equation}
 \begin{equation} \label{e7}
\frac{D^\beta}{Dt^\beta}\textbf{A}=D^\beta_t\textbf{A}+(\textbf{V}\cdot\nabla)\textbf{A}-\textbf{L}\textbf{A}-\textbf{A}\textbf{L}^T,
\end{equation}
where $D^\alpha_t$ and $D^\beta_t$ are the Caputo fractional derivative operators of order $\alpha$ and $\beta$ with respect to $t$,
respectively.
%While $\alpha=\beta=1$, Eq.(4) may hold for an ordinary Oldroyd-B fluid. For $\lambda=0$ we have a generalized second grade fluid, and for
$\theta=0$ we have a generalized Maxwell fluid. The classical Navier-Stokes fluid can be obtained for $\lambda=\theta=0$ and $\alpha=\beta=1$.
The Hall effect is taken into consideration, and thus we have
\begin{equation} \label{e8}
\textbf{J}\times\frac{\omega_e\tau_e}{B_0}(\textbf{J}\times\textbf{B})=
\sigma[\textbf{E}+\textbf{V}\times\textbf{B}\times\frac{1}{en_e}\bigtriangledown p_e],
\end{equation}
in which $\omega_e$ is the cyclotron frequency of electrons, $\tau_e$ is the electron collision time, $\sigma$ is the electrical conductivity,
$e$ is the electron charge, $n_e$ is the number density of electrons and $p_e$ is the electron pressure, $\textbf{E}$ is the electric field.
Further, it is assumed that there is no applied or polarization voltage so that $\textbf{E}=0$.
\par Moreover, Darcy resistance $\textbf{r}$ can also be interpreted as measure of the flow resistance offered by the solid matrix, thus
$\textbf{r}$ satisfies the following equation:
\begin{equation} \label{e9}
(1+\lambda\frac{D^\alpha}{D t^\alpha})\textbf{r}=-\frac{\mu\phi_1}{k}(1+\theta\frac{D^\beta}{D t^\beta})\textbf{V}.
\end{equation}
where $\phi_1$ is the porosity of the medium.
\par Substituting (6) and (7) into (5) and taking into account the initial condition $\textbf{S}(x,y,0)=0$, we obtain $S_{xx}=S_{xy}=S_{yy}=0$
and the relevant partial differential equations
\begin{equation} \label{e10}
(1+\lambda D_t^\alpha) S_{zx}=\mu(1+\theta D_t^\beta)\frac{\partial u}{\partial x},~(1+\lambda D_t^\alpha)S_{zy}=\mu(1+\theta
D_t^\beta)\frac{\partial u}{\partial y},
\end{equation}
then substituting Eqs.(8-10) into Eq.(4) and neglecting the pressure gradient, then we obtain the generalized two dimensional Oldroyd-B fluid
with the influence of Hall current.

\section{Preliminary}

\par Firstly, in the $x-$direction $[0,L_x]$, we take the mesh points $x_i=ih_x$, $i=0, 1,\cdots,M_x$, in the $y-$direction $[0,L_y]$, we
take the mesh points $y_j=jh_y$, $j=0, 1,\cdots,M_y$, and $t_n = n\tau$, $n = 0, 1,\cdots, N$, where $h_x =L_x/M_x$,$h_y =L_y/M_y$, $ \tau=
T/N$ are the uniform spatial step size and temporal step size, respectively. Denote $\Omega_{\tau}\equiv \{t_n|~0\leq n\leq N\}$,
$\Omega_{h}\equiv \{(x_i,y_j)|~0\leq i\leq M_x, 0\leq j\leq M_y\}$. Suppose $u_{ij}^n=u(x_i,y_j,t_n)$, $u_{ij}^n$ is a grid function on
$\Omega_{h}\times\Omega_{\tau}$. We introduce the following notations:
\begin{align*}
\nabla_tu_{ij}^{n}=\frac{u_{ij}^n-u_{ij}^{n-1}}{\tau},\quad
u_{ij}^{n-\frac{1}{2}}=\frac{u_{ij}^n+u_{ij}^{n-1}}{2},\quad\nabla_xu_{ij}^{n}=\frac{u_{ij}^n-u_{i-1j}^{n}}{h},\\
\delta_x^2u_{ij}^n=\frac{u_{i-1j}^n-2u_{ij}^n+u_{i+1j}^n}{h_x^2},\quad\delta_y^2u_{ij}^n=\frac{u_{ij-1}^n-2u_{ij}^n+u_{ij+1}^n}{h_y^2}.
\end{align*}
\par For improving the temporal convergence accuracy from first order \cite{Fengf} to higher order, we will give a new discretization of the fractional derivative ${D_t^\alpha}u$ $(0<\alpha<1)$.
%\par and according to the results
%of $L2-1\sigma$ formula in \cite{Alikhanov1}, here we take $\sigma=\frac{1}{2}$ as a special case of $L2-1\sigma$ formula, then we will give a
%new discretization of the time fractional derivative ${D_t^\alpha}u$ $(0<\alpha<1)$, which will help us improve the temporal order.
For the function $u\in C^{0,0,3}_{x,y,t}(\Omega\times(0,T])$ at the point$(x_i,y_j,t_{n-\frac{1}{2}})$, then the following equality holds
\begin{align}
{D_t^\alpha}u(\cdot,t_{n-\frac{1}{2}})&=\frac{1}{\Gamma(1-\alpha)}
\int_{0}^{t_{n-\frac{1}{2}}}\frac{\partial_\eta u(\cdot,\eta)d\eta}{(t_{n-\frac{1}{2}}-\eta)^\alpha}\nonumber\\\label{e11}
&=\frac{1}{\Gamma(1-\alpha)}\sum_{k=1}^{n-1}\int_{t_{k-1}}^{t_{k}}\frac{\partial_\eta
u(\cdot,\eta)d\eta}{(t_{n-\frac{1}{2}}-\eta)^\alpha}\nonumber\\
&+\frac{1}{\Gamma(1-\alpha)}\int_{t_{n-1}}^{t_{n-\frac{1}{2}}}\frac{\partial_\eta u(\cdot,\eta)d\eta}{(t_{n-\frac{1}{2}}-\eta)^\alpha},
\end{align}
where $u(\cdot,t)$ denotes $u(x_i,y_j,t).$ On each interval $[t_{k-1},t_k](1\leq k\leq n-1)$, denoting the quadratic interpolation
$\Pi_{2,k}u(\cdot,t)$ of $u(\cdot,t)$ using three points $(t_{k-1},u(\cdot,t_{k-1}))$, $(t_{k},u(\cdot,t_{k}))$ and
$(t_{k+1},u(\cdot,t_{k+1}))$, we get
\begin{align}
\Pi_{2,k}u(\cdot,t)&= u(\cdot,t_{k-1})\frac{(t-t_k)(t-t_{k+1})}{2\tau^2}-u(\cdot,t_{k})
\frac{(t-t_{k-1})(t-t_{k+1})}{\tau^2}\nonumber\\\label{e12}
&+u(\cdot,t_{k+1})\frac{(t-t_{k-1})(t-t_{k})}{2\tau^2},
\end{align}
on the interval $[t_{n-1},t_{n-\frac{1}{2}}]$, denoting the linear interpolation $\Pi_{1,n}u(\cdot,t)$ of $u(\cdot,t)$ using two points
$(t_{n-1},u(\cdot,t_{n-1}))$ and $(t_{n},u(\cdot,t_{n}))$, we get
\begin{equation} \label{e13}
\Pi_{1,n}u(\cdot,t)= u(\cdot,t_{n})\frac{(t-t_n)}{\tau}-u(\cdot,t_{n-1})
\frac{(t-t_{n-1})}{\tau},
\end{equation}
similar to the calculation in \cite{Alikhanov1}, we obtain the difference analog of the Caputo fractional derivative of the order
$\alpha~(0<\alpha<1)$ for the function $u(x,y,t)$ in the following form:
\begin{equation}
   \begin{aligned}
D_t^{\alpha}u(\cdot,t_{n-\frac{1}{2}}) &=\frac{1}{\Gamma(1-\alpha)}\sum_{k=1}^{n-1}\int_{t_{k-1}}^{t_{k}}
\frac{\partial_\eta(\Pi_{2,k}u(\cdot,\eta))d\eta}{(t_{n-\frac{1}{2}}-\eta)^\alpha}\\\label{e14}
&+\frac{1}{\Gamma(1-\alpha)}\int_{t_{n-1}}^{t_{n-\frac{1}{2}}}
\frac{\partial_\eta(\Pi_{1,n}u(\cdot,\eta))d\eta}{(t_{n-\frac{1}{2}}-\eta)^\alpha}+R_1\\
   & =\frac{\tau^{-\alpha}}{\Gamma(2-\alpha)}
\Big[c_0^{(\alpha)}u_{ij}^n-\sum_{k=1}^{n-1}(c_{n-k-1}^{(\alpha)}-c_{n-k}^{(\alpha)})u_{ij}^k-
c_{n-1}^{(\alpha)}u_{ij}^0\Big]+R_1\\
   & =\frac{\tau^{1-\alpha}}{\Gamma(2-\alpha)}\sum^n_{k=1}c_{n-k}^{(\alpha)}\nabla_tu_{ij}^k+R_1,
    \end{aligned}
\end{equation}
where $R_1$ is the error, and for $n=1$
\begin{equation} \label{e15}
c_{0}^{(\alpha)}=a_0^{(\alpha)},
\end{equation}
for $n\geq2$,
\begin{equation} \label{e16}
 c_{k}^{(\alpha)}=\left\{
   \begin{aligned}
   & a_0^{(\alpha)}+b_1^{(\alpha)}, \qquad\qquad k=0,\\
   & a_k^{(\alpha)}+b_{k+1}^{(\alpha)}-b_{k}^{(\alpha)}, ~~1\leq k\leq n-2,\\
   & a_k^{(\alpha)}-b_{k}^{(\alpha)}, \qquad\qquad k=n-1,
    \end{aligned}
  \right.
\end{equation}
where $
  a_0^{(\alpha)} =(\frac{1}{2})^{1-\alpha}, a_k^{(\alpha)}=(k+\frac{1}{2})^{1-\alpha}-(k-\frac{1}{2})^{1-\alpha},k\geq1,$
and $ b_k^{(\alpha)} =\frac{1}{2-\alpha}\Big[(k+\frac{1}{2})^{2-\alpha}-(k-\frac{1}{2})^{2-\alpha}\Big]
-\frac{1}{2}\Big[(k+\frac{1}{2})^{1-\alpha}+(k-\frac{1}{2})^{1-\alpha}\Big],k\geq1.$

\begin{lemma}\label{lm1}
For any $0<\alpha<1$ and $u\in C^{0,0,3}_{x,y,t}(\Omega\times(0,T])$, the error

\begin{equation} \label{e17}
|R_1|=|D_t^{\alpha}u(\cdot,t_{n-\frac{1}{2}})
-\frac{\tau^{1-\alpha}}{\Gamma(2-\alpha)}\sum^n_{k=1}c_{n-k}^{(\alpha)}\nabla_tu_{ij}^k|=O(\tau^{2-\alpha}),
\end{equation}

\end{lemma}
\begin{proof}
Let $R_1=R_1^{n-1}+R_{n-1}^{n-\frac{1}{2}}$, where
\begin{align*}
R_1^{n-1}&=\frac{1}{\Gamma(1-\alpha)}\sum_{k=1}^{n-1}\int_{t_k}^{t_{k+1}}
\frac{\partial_\eta(u(\cdot,\eta)-\Pi_{2,k}u(\cdot,\eta)) d\eta}{(t_{n-\frac{1}{2}}-\eta)^\alpha},\\
R_{n-1}^{n-\frac{1}{2}}&=\frac{1}{\Gamma(1-\alpha)}\int_{t_{n-1}}^{t_{n-\frac{1}{2}}}
\frac{\partial_\eta(u(\cdot,\eta)-\Pi_{1,n}u(\cdot,\eta))d\eta}{(t_{n-\frac{1}{2}}-\eta)^\alpha},
\end{align*}
we estimate the error $R_1^{n-1}$ similarly to \cite{Alikhanov1}:
$$R_1^{n-1}\leq \frac{2^\alpha|\partial_t^{3}u(\cdot,\xi)|}{3\Gamma(1-\alpha)}\tau^{3-\alpha},$$
and
\begin{align*}
R_{n-1}^{n-\frac{1}{2}}&=\frac{1}{\Gamma(1-\alpha)}\int_{t_{n-1}}^{t_{n-\frac{1}{2}}}
\frac{(\partial_\eta u(\cdot,\eta)-\frac{u(\cdot,t_n)-u(\cdot,t_{n-1})}{\tau})d\eta}{(t_{n-\frac{1}{2}}-\eta)^\alpha}\\
&=\frac{1}{\Gamma(1-\alpha)}\int_{t_{n-1}}^{t_{n-\frac{1}{2}}}
\frac{(\partial_\eta u(\cdot,t_{n-\frac{1}{2}})-\frac{u(\cdot,t_n)-u(\cdot,t_{n-1})}{\tau})d\eta}{(t_{n-\frac{1}{2}}-\eta)^\alpha}
+O(\tau^{2-\alpha})\\
&=O(\tau^{2-\alpha}),
\end{align*}
then the error $|R_1|=O(\tau^{2-\alpha})$, Lemma 1 is proved.
\end{proof}
The new discretization of ${D_t^\alpha}u$ $(0<\alpha<1)$ will help us improve the temporal convergence accuracy of the Eq.({\ref{e1}}) from $\tau$ \cite{Fengf} to
$\tau^{\min\{3-\gamma_l,2-\alpha_m,2-\beta_r\}}$.

\begin{lemma} \label{lm2}

The coefficients $a_k^{(\alpha)}(k=0,1,2,\ldots,n)$ and $b_k^{(\alpha)}(k=1,2,\ldots,n)$ satisfy the following properties:
\begin{itemize}
\item[(1)] $a_k^{(\alpha)}>0,\lim\limits_{k\rightarrow\infty}a_k^{(\alpha)}=0$;~ $a_k^{(\alpha)}>a_{k+1}^{(\alpha)}, k\geq1$;~
    $a_{k+1}^{(\alpha)}-2a_k^{(\alpha)}+a_{k-1}^{(\alpha)}\geq0,$ $k\geq2$,
\item[(2)] $b_{k}^{(\alpha)}>0;$ ~$\lim\limits_{k\rightarrow\infty}b_k^{(\alpha)}=0$;~ $b_{k}^{(\alpha)}>b_{k+1}^{(\alpha)}.$
\end{itemize}

\begin{proof}
Similar to the proof in \cite{Alikhanov1}.
\end{proof}

\end{lemma}

\begin{lemma} \label{lm3}

The coefficients $c_k^{(\alpha)}(k=0,1,2,\ldots,n-1)$ satisfy the following properties:
\begin{itemize}
\item[(1)] $c_k^{(\alpha)}>0$, $\lim\limits_{k\rightarrow\infty}c_k^{(\alpha)}=0$,
\item[(2)] $c_0^{(\alpha)}>c_1^{(\alpha)}>\cdots>c_{n-2}^{(\alpha)}>c_{n-1}^{(\alpha)}$,
\item[(3)] $c_{k+1}^{(\alpha)}-2c_k^{(\alpha)}+c_{k-1}^{(\alpha)}>0,~k\geq2,$$~k\neq n-2.$
%\item[(4)] $c_{n}^{(n+1,\alpha)}>c_{n+1}^{(n+2,\alpha)}, n\geq1$.
\end{itemize}

\begin{proof}
\par Similar to the proof in \cite{Alikhanov1}, the properties (1)-(2) are easy to get. Let us prove the property (3).
\par For $k\geq2$ and $k+1\neq n-1$ we have
 \begin{align*}
 &{c_{k+1}^{(\alpha)}-2c_k^{(\alpha)}+c_{k-1}^{(\alpha)}=
 (a_{k+1}^{(\alpha)}+b_{k+2}^{(\alpha)}-b_{k+1}^{(\alpha)})-2(a_{k}^{(\alpha)}+b_{k+1}^{(\alpha)}-b_{k}^{(\alpha)})+(a_{k-1}^{(\alpha)}+b_{k}^{(\alpha)}-b_{k-1}^{(\alpha)})}\\
 &=\frac{1}{2-\alpha}
 \big[(k+\frac{5}{2})^{2-\alpha}-4(k+\frac{3}{2})^{2-\alpha}+6(k+\frac{1}{2})^{2-\alpha}
 -4(k-\frac{1}{2})^{2-\alpha}+(k-\frac{3}{2})^{2-\alpha}\big]\\
 &-\frac{1}{2}
 \big[(k+\frac{5}{2})^{1-\alpha}-4(k+\frac{3}{2})^{1-\alpha}+6(k+\frac{1}{2})^{1-\alpha}
 -4(k-\frac{1}{2})^{1-\alpha}+(k-\frac{3}{2})^{1-\alpha}\big]\\
 &=\alpha(1-\alpha)(1+\alpha)\int_0^1dz_1\int_0^1dz_2\int_0^1dz_3\int_0^1\frac{dz_4}{(k-\frac{3}{2}+z_1+z_2+z_3+z_4)^{2+\alpha}}\\
 &+\frac{\alpha(1-\alpha)(1+\alpha)(2+\alpha)}{2}\int_0^1dz_1\int_0^1dz_2\int_0^1dz_3\int_0^1\frac{dz_4}{(k-\frac{3}{2}+z_1+z_2+z_3+z_4)^{3+\alpha}}\\
 &>\alpha(1-\alpha)(1+\alpha)\frac{1}{(k+\frac{5}{2})^{2+\alpha}}+\frac{\alpha(1-\alpha)(1+\alpha)(2+\alpha)}{2}\frac{1}{(k+\frac{5}{2})^{3+\alpha}}>0,
 \end{align*}
%\par For property (4), when $n\geq1$,
%\begin{align*}
%c_{n}^{(n+1,\alpha)}&=a_n^\alpha-b_n^\alpha,\\
%&=\frac{1}{2-\alpha}\big[(n-\frac{1}{2})^{2-\alpha}-(n+\frac{1}{2})^{2-\alpha}\big]
%+\frac{1}{2}\big[3(n+\frac{1}{2})^{1-\alpha}-(n-\frac{1}{2})^{1-\alpha}\big],
%\end{align*}
%Let $h(z)=\frac{1}{2-\alpha}\big[(z-\frac{1}{2})^{2-\alpha}-(z+\frac{1}{2})^{2-\alpha}\big]
%+\frac{1}{2}\big[3(z+\frac{1}{2})^{1-\alpha}-(z-\frac{1}{2})^{1-\alpha}\big],$ for $z\geq1$, it follows
%\begin{align*}
%h'(z)=&(z-\frac{1}{2})^{1-\alpha}-(z+\frac{1}{2})^{1-\alpha}
%+\frac{3}{2}(1-\alpha)(z+\frac{1}{2})^{-\alpha}-\frac{1}{2}(1-\alpha)(z-\frac{1}{2})^{-\alpha}\\
%=&-(1-\alpha)\int_0^1\frac{d\xi}{(z-\frac{1}{2}+\xi)^\alpha}
%+(1-\alpha)(z+\frac{1}{2})^{-\alpha}\\
%&+\frac{1}{2}(1-\alpha)(z+\frac{1}{2})^{-\alpha}-\frac{1}{2}(1-\alpha)(z-\frac{1}{2})^{-\alpha},
%\end{align*}
%according to the monotone decreasing property of function $(z-\frac{1}{2}+\xi)^{-\alpha}$ with respect to $\xi$ on $[0,1],$ so the above
formula $h'(z)<0$. Consequently, $c_{n}^{(n+1,\alpha)}>c_{n+1}^{(n+2,\alpha)}, n\geq1$.
the proof is completed.
\end{proof}
\end{lemma}
\textbf{Remark 1:} When $c_{2}^{(\alpha)}-2c_1^{(\alpha)}+c_{0}^{(\alpha)}$ is a formula of parameter $\alpha$, but it is not
 always greater than 0, by simple algebraic calculation, when $0<\alpha<0.4471$,\quad $c_{2}^{(\alpha)}-2c_1^{(\alpha)}+c_{0}^{(\alpha)}<0$, else
$c_{2}^{(\alpha)}-2c_1^{(\alpha)}+c_{0}^{(\alpha)}>0$.
\par Similarly, for the time fractional derivative ${D_t^\beta}u$ $(0<\beta<1)$, we use the following formula
\begin{equation}
\begin{aligned}\label{e18}
{D_t^\beta}u(\cdot,t_{n-\frac{1}{2}})&=\frac{\tau^{-\beta}}{\Gamma(2-\beta)}
\Big[c_0^{(\beta)}u_{ij}^n
-\sum_{k=1}^{n-1}(c_{n-k-1}^{(\beta)}-c_{n-k}^{(\beta)})u_{ij}^k-
c_{n-1}^{(\beta)}u_{ij}^0\Big]+R_2,\\
&=\frac{\tau^{1-\beta}}{\Gamma(2-\beta)}
\sum^n_{k=1}c_{n-k}^{(\beta)}\nabla_tu_{ij}^k
+R_2,
\end{aligned}
\end{equation}
where $R_2=O(\tau^{2-\beta})$ and the coefficients are similar to Lemma 2 and Lemma 3.

\par Next, we will give the new scheme for the fractional derivative ${D_t^\beta}(\partial _x^2+\partial_y^2)u,$ $0<\beta<1$.
\par Since
\begin{align*}
{D_t^\beta}(\partial _x^2+\partial_y^2)u(\cdot,t_{n-\frac{1}{2}})&=(\delta_x^2+\delta_y^2)u(\cdot,t_{n-\frac{1}{2}})\\
&-\frac{h_x^2}{12}\frac{\partial^4u(\xi_i,y_j,t_{n-\frac{1}{2}})}{\partial x^4}
-\frac{h_y^2}{12}\frac{\partial^4u(x_i,\varsigma_j,t_{n-\frac{1}{2}})}{\partial y^4},
\end{align*}
where $x_{i-1}\leq \xi_i\leq x_i, y_{j-1}\leq \varsigma_j\leq y_j$, then we have
\begin{align*}
{D_t^\beta}(\partial _x^2+\partial_y^2)u(\cdot,t_{n-\frac{1}{2}})&=\frac{\tau^{-\beta}}{\Gamma(2-\beta)}
\Big[c_0^{(\beta)}(\partial _x^2+\partial_y^2)u(\cdot,t_{n})\\
&-\sum_{k=1}^{n-1}(c_{n-k-1}^{(\beta)}-c_{n-k}^{(\beta)})(\partial _x^2+\partial_y^2)u(\cdot,t_{k})\\
&-c_{n-1}^{(\beta)}(\partial _x^2+\partial_y^2)u(\cdot,0)\Big]+R_3,\\
& =\frac{\tau^{-\beta}}{\Gamma(2-\beta)}
\Big[c_0^{(\beta)}(\delta_x^2+\delta_y^2)u(\cdot,t_{n})\\
&-\sum_{k=1}^{n-1}(c_{n-k-1}^{(\beta)}-c_{n-k}^{(\beta)})(\delta_x^2+\delta_y^2)u(\cdot,t_{k})\\
&-c_{n-1}^{(\beta)}(\delta_x^2+\delta_y^2)u(\cdot,0)\Big]+R_4+R_3,
\end{align*}
where $|R_3|\leq C\tau^{2-\beta}$ and
\begin{align*}
R_4&=-\frac{h_x^2\tau^{-\beta}}{12\Gamma(2-\beta)}\sum_{k=0}^{n-1}c_{k}^{(\beta)}
\Big[\frac{\partial^4u(\xi_i,y_j,t_{n-k})}{\partial x^4}-\frac{\partial^4u(\xi_i,y_j,t_{n-k-1})}{\partial x^4}\Big]\\
&-\frac{h_y^2\tau^{-\beta}}{12\Gamma(2-\beta)}\sum_{k=0}^{n-1}c_{k}^{(\beta)}
\Big[\frac{\partial^4u(x_i,\varsigma_j,t_{n-k})}{\partial y^4}-\frac{\partial^4w(x_i,\varsigma_j,t_{n-k-1})}{\partial y^4}\Big]\\
&=-\frac{h_x^2\tau^{1-\beta}}{12\Gamma(2-\beta)}
\sum_{k=0}^{n-1}c^{(\beta)}_{k}\frac{\partial^5u(\xi_i,y_j,\eta_{n-k})}{\partial x^4\partial t}-\frac{h_y^2\tau^{1-\beta}}{12\Gamma(2-\beta)}
\sum_{k=0}^{n-1}c^{(\beta)}_{k}\frac{\partial^5u(x_i,\varsigma_j,\upsilon_{n-k})}{\partial y^4\partial t},
\end{align*}
where $t_{n-k-1}<\eta_{n-k},\upsilon_{n-k}<t_{n-k}$.
\begin{align*}
|R_4|&\leq \frac{\tau^{1-\beta}}{12\Gamma(2-\beta)}\big(h_x^2
\max_{\substack{(x,y)\in\Omega\\t\in(0,T]}}\bigg|\frac{\partial^5u(\xi_i,y_j,\eta_{n-k})}{\partial x^4\partial
t}\bigg|+h_y^2\max_{\substack{(x,y)\in\Omega\\t\in(0,T]}}\bigg|\frac{\partial^5u(x_i,\varsigma_j,\upsilon_{n-k})}{\partial y^4\partial
t}\bigg|\big)
\sum_{k=0}^{n-1}c_{k}^{(\beta)}\\
&=\frac{\tau^{1-\beta}}{12\Gamma(2-\beta)}\big(h_x^2
\max_{\substack{(x,y)\in\Omega\\t\in(0,T]}}\bigg|\frac{\partial^5u(\xi_i,y_j,\eta_{n-k})}{\partial x^4\partial
t}\bigg|+h_y^2\max_{\substack{(x,y)\in\Omega\\t\in(0,T]}}\bigg|\frac{\partial^5u(x_i,\varsigma_j,\upsilon_{n-k})}{\partial y^4\partial
t}\bigg|\big)
\sum_{k=0}^{n-1}a_{k}^{(\beta)}\\
&=\frac{\tau^{1-\beta}}{12\Gamma(2-\beta)}\big(h_x^2
\max_{\substack{(x,y)\in\Omega\\t\in(0,T]}}\bigg|\frac{\partial^5u(\xi_i,y_j,\eta_{n-k})}{\partial x^4\partial
t}\bigg|+h_y^2\max_{\substack{(x,y)\in\Omega\\t\in(0,T]}}\bigg|\frac{\partial^5u(x_i,\varsigma_j,\upsilon_{n-k})}{\partial y^4\partial
t}\bigg|\big)
\cdot (n-\frac{1}{2})^{1-\beta}\\
&\leq \frac{T^{1-\beta}}{12\Gamma(2-\beta)}\big(h_x^2
\max_{\substack{(x,y)\in\Omega\\t\in(0,T]}}\bigg|\frac{\partial^5u(\xi_i,y_j,\eta_{n-k})}{\partial x^4\partial
t}\bigg|+h_y^2\max_{\substack{(x,y)\in\Omega\\t\in(0,T]}}\bigg|\frac{\partial^5u(x_i,\varsigma_j,\upsilon_{n-k})}{\partial y^4\partial
t}\bigg|\big),
\end{align*}
then we have
\begin{equation}\label{e19}
   \begin{aligned}
{D_t^\beta}(\partial
_x^2+\partial_y^2)u(\cdot,t_{n-\frac{1}{2}})=\frac{\tau^{1-\beta}}{\Gamma(2-\beta)}\sum^n_{k=1}c_{n-k}^{(\beta)}\nabla_t(\delta_x^2+\delta_y^2)u_{ij}^k
+R_5,
    \end{aligned}
\end{equation}
where $R_5\leq C(\tau^{2-\beta}+h_x^2+h_y^2)$.

\begin{lemma} \label{lm4}
For $0<\alpha<1$, $c_k^{(\alpha)}$ is defined as (15),(16), and for any positive integer $N$ and real vector
$\mathbf{Q}=(v^1,v^2,\ldots,v^{N-1},v^{N})\in R^{N+1}$, we have
\begin{equation} \label{e20}
\sum_{n=1}^{N}\sum_{k=1}^{n} c_{n-k}^{(\alpha)}\,v^kv^n\geq0.
\end{equation}

Analyse: This is the vital Lemma, which plays an important role in the proof of unconditional stability. In the literature \cite{Lopez}, the
real sequence ${a_0,a_1,\cdots,a_n,\cdots}$ satisfy $a_{n+1}-2a_n+a_{n-1}\geq 0, n=1,2,\cdots$. But from the \textbf{Remark 1},
the coefficients $c_{2}^{(\alpha)}-2c_1^{(\alpha)}+c_{0}^{(\alpha)}$ is not always greater than $0$. Next, we will give the proof that it does not
affect the results. Then similar to Proposition 5.1 \cite{Lopez}, we will first prove the following results.
\par For the coefficients $c_k^{(\alpha)}$ is defined as (16),(17), if
$$\lim\limits_{M\rightarrow\infty}c_{M+1}^{(\alpha)}\sum\limits_{k=0}\limits^{M}\cos (kx)\geq0,$$
 (the limit may be $+\infty$), then
 $$\sum\limits_{k=0}\limits^{\infty}c_k^{(\alpha)}\cos (kx)>0$$
 (and the sum of the series may be $+\infty$).
\begin{proof}
For each nonnegative integer $M$, we define $A_M=\sum\limits_{k=0}\limits^{M}c_k^{(\alpha)}\cos (kx)$, by summing by parts twice, we arrive at
\begin{align*}
A_M&=c_{M+1}^{(\alpha)}\sum\limits_{k=0}\limits^{M}\cos (kx)-
(c_{M+2}^{(\alpha)}-c_{M+1}^{(\alpha)})\sum\limits_{k=0}\limits^{M}\sum\limits_{p=0}\limits^{k}\cos (px)\\
&+\sum\limits_{k=0}\limits^{M}(c_{k+2}^{(\alpha)}-2c_{k+1}^{(\alpha)}+c_{k}^{(\alpha)})\\
&+\sum\limits_{k=1}\limits^{M}(c_{k+2}^{(\alpha)}-2c_{k+1}^{(\alpha)}+c_{k}^{(\alpha)})
\sum\limits_{p=1}\limits^{k}\sum\limits_{m=0}\limits^{p}\cos (mx).
\end{align*}
\par From the well-known identity
\begin{equation*}
\sum\limits_{m=0}\limits^{p}\cos (mx)=\left\{
   \begin{aligned}
   & \frac{1-\cos x+\cos(px)-\cos(p
   +1)x}{2(1-\cos x)},~\cos x\neq1,\\
   & p+1, \qquad\qquad\qquad\qquad~~~~~~~~~~~~~~\cos x=1.
    \end{aligned}
  \right.
\end{equation*}
by induction we get
\begin{align*}
\sum\limits_{p=0}\limits^{k}\sum\limits_{m=0}\limits^{p}\cos (mx)&=\left\{
   \begin{aligned}
   & \frac{(k+2)-(k+1)\cos x-\cos(k
   +1)x}{2(1-\cos x)},~\cos x\neq1\\
   & \frac{(k+2)(k+1)}{2}, \qquad\qquad\qquad\qquad~~~~~~\cos x=1
    \end{aligned}
  \right.\\
  &\geq0,
\end{align*}
and
\begin{align*}
\sum\limits_{p=1}\limits^{k}\sum\limits_{m=0}\limits^{p}\cos (mx)&=\left\{
   \begin{aligned}
   & \frac{k-(k-1)\cos x-\cos(k
   +1)x}{2(1-\cos x)},~\cos x\neq1\\
   & \frac{k(k+3)}{2}, \qquad\qquad\qquad\qquad~~~~~~\cos x=1
    \end{aligned}
  \right.\\
   &\geq0.
\end{align*}
\par From the properties of $c_k^{(\alpha)}$ in Lemma 3,
\begin{align*}
\lim\limits_{M\rightarrow\infty}\sum\limits_{k=0}\limits^{M}(c_{k+2}^{(\alpha)}-2c_{k+1}^{(\alpha)}+c_{k}^{(\alpha)})
&=\lim\limits_{M\rightarrow\infty}(c_{0}^{(\alpha)}-c_{1}^{(\alpha)}+c_{M+2}^{(\alpha)}-c_{M+1}^{(\alpha)})\\
&=c_{0}^{(\alpha)}-c_{1}^{(\alpha)}>0,
\end{align*}
it is clear that $\sum\limits_{k=0}\limits^{\infty}c_k^{(\alpha)}\cos (kx)>0.$
\end{proof}
 Next one makes the same proof to Proposition 5.2 \cite{Lopez}, (21) is proved.

\end{lemma}

Next we give the following lemmas \cite{LZL15}.

\begin{lemma} \label{lm5}
\par Discretization of the time fractional derivative ${D_t^\gamma}u(x,y,t)$ $(1<\gamma<2)$. From the results in \cite{LZL15}, at mesh points
$(x_i,y_j,t_{n-\frac{1}{2}})$ we get
\begin{equation}\label{e21}
\begin{aligned}
{D_t^\gamma}u(x_i,y_j,t_{n-\frac{1}{2}})&\approx\frac{1}{2}[D_t^{\gamma}u(x_i,y_j,t_{n})+D_t^{\gamma}u(x_i,y_j,t_{n-1})\big]\\
&=\frac{\tau^{1-\gamma}}{\Gamma(3-\gamma)}
\Big[a_0^{(\gamma)}\nabla_tu_{ij}^n-\sum_{k=1}^{n-1}(a_{n-k-1}^{(\gamma)}-a_{n-k}^{(\gamma)})\nabla_tu_{ij}^k\\
&-a_{n-1}^{(\gamma)}\frac{\partial u(x_i,y_j,0)}{\partial t}\Big]+R_8,
\end{aligned}
\end{equation}
where $R_8=O(\tau^{3-\gamma})$, the coefficients $a_k^{(\gamma)}=(k+1)^{2-\gamma}-k^{2-\gamma}$, $k=0,1,2,\ldots$ and satisfy the following
properties:
\begin{itemize}
\item[(1)] $a_k^{(\gamma)}>0$, $a_0^{(\gamma)}=1$, $a_k^{(\gamma)}>a_{k+1}^{(\gamma)}$, ${\lim\limits_{k\to \infty}}a_k^{(\gamma)}=0$,
\item[(2)] $\sum\limits_{k=0}^{n-1}(a_k^{(\gamma)}-a_{k+1}^{(\gamma)})+a_{n}^{(\gamma)}=1$,
\item[(3)] $(2-\gamma)(k+1)^{1-\gamma}\leq a_k^{(\gamma)}\leq(2-\gamma)k^{1-\gamma}$.
\end{itemize}
\end{lemma}
\begin{lemma}[\cite{Fengf}]\label{lm6}
For $1<\gamma<2$, define $a_k^{(\gamma)}=(k+1)^{2-\gamma}-k^{2-\gamma}$, $k=0,1,2,\ldots,n$ and $S=\{S_1,S_2,S_3,\ldots\}$ and $P$, then it
holds that
\begin{align*}
&\frac{\tau^{1-\gamma}}{\Gamma(3-\gamma)}\sum_{n=1}^{N}\Big[a_0^{(\gamma)}S_n-
\sum_{k=1}^{n-1}(a_{n-k-1}^{(\gamma)}-a_{n-k}^{(\gamma)})S_k-
a_{n-1}^{(\gamma)}P\Big]S_n\\
\geq& \frac{T^{1-\gamma}}{2\Gamma(2-\gamma)}\sum_{n=1}^{N}S_n^2
-\frac{T^{2-\gamma}}{2\tau\Gamma(3-\gamma)}P^2,~N=1,2,3,\ldots
\end{align*}
\end{lemma}
\section{The derivation of the difference scheme}

\par To develop a finite difference scheme for the generalized problem (1), we define $\mathcal{V}_h=\{v~|~v~
\text{is~a~grid~function~on}~\Omega_{h}~\text{and}~v_{0j}=v_{M_xj}=v_{i0}=v_{iM_y}=0\}.$
For any $u,v\in \mathcal{V}_h$, we define the following discrete inner products and induced norms:
\begin{align*}
&(u,v)=h_xh_y\sum_{i=1}^{M_x-1}\sum_{i=1}^{M_y-1}u_{ij}v_{ij},\\
&\langle(\nabla_x+\nabla_y)u,(\nabla_x+\nabla_y)v\rangle=h_xh_y\sum_{i=1}^{M_x}\sum_{i=1}^{M_y}(\nabla_x+\nabla_y)u_{ij}\cdot(\nabla_x+\nabla_y)v_{ij},\\
&||v||_0=\sqrt{(v,v)},\quad ||v||_\infty=\max_{\substack{1 \leq i \leq M_x \\ 1\leq j\leq M_y}}{\mid v\mid},\quad
|v|_1=\sqrt{\langle(\nabla_x+\nabla_y)v,(\nabla_x+\nabla_y)v\rangle},
\end{align*}
and
\begin{equation} \label{e22}
||v||_{1}=\sqrt{d_2||v||_0^2+d_3||v||_1^2}.
\end{equation}
Similar to one-dimensional case, it is easy to get the following properties
\begin{align*}
((\delta_x^2+\delta_y^2)v^k,v^{n})=-\langle(\nabla_x+\nabla_y)v^k,(\nabla_x+\nabla_y)v^n\rangle,
\end{align*}
\begin{align} \label{e23}
((\delta_x^2+\delta_y^2)v^{k},\nabla_tv^{n})
&=-\frac{1}{\tau}\langle(\nabla_x+\nabla_y)v^k,(\nabla_x+\nabla_y)v^n-(\nabla_x+\nabla_y)v^{n-1}\rangle\\\label{e24}
%=-\langle(\nabla_x+\nabla_y)v^k,\nabla_t((\nabla_x+\nabla_y)v^{n})\rangle.}
(\nabla_t((\delta_x^2+\delta_y^2)v^{k}),\nabla_tv^{n})
&=-\langle\nabla_t((\nabla_x+\nabla_y)v^k),\nabla_t((\nabla_x+\nabla_y)v^{n})\rangle.
\end{align}

\par We define the grid function$f_{ij}^n=f(x_i,y_j,t_n),\varphi_{ij}=\varphi(x_i,y_j),\phi_{ij}=\phi(x_i,y_j),$
where $(x_i,y_j) \in \Omega_h,0\leq n\leq N.$
\par Now, we will present the difference scheme for the two-dimensional problem (\ref{e1}), assume that
$u(x,y,t)\in C_{x,y,t}^{4,4,3}(\Omega\times(0,T])$,we have
\begin{align}\nonumber
&{\sum_{l=1}^pa_lD^{\gamma_l}_tu(x_i,y_j,t_{n-\frac{1}{2}})+b_1\frac{\partial u(x_i,y_j,t_{n-\frac{1}{2}})}{\partial
t}+\sum_{m=1}^qc_mD^{\alpha_m}_tu(x_i,y_j,t_{n-\frac{1}{2}})}\\\label{e26}
&+b_2u(x_i,y_j,t_{n-\frac{1}{2}})
=b_3(\frac{\partial^2 u(x_i,y_j,t_{n-\frac{1}{2}})}{\partial x^2}+\frac{\partial^2 u(x_i,y_j,t_{n-\frac{1}{2}})}{\partial y^2})\\\nonumber
&+\sum_{r=1}^sd_rD^{\beta_r}_t(\frac{\partial^2 u(x_i,y_j,t_{n-\frac{1}{2}})}{\partial x^2}+\frac{\partial^2
u(x_i,y_j,t_{n-\frac{1}{2}})}{\partial y^2})+f(x_i,y_j,t_{n-\frac{1}{2}}).
\end{align}
From Eqs.(\ref{e14}), (\ref{e19}) and (\ref{e21}), we have the following scheme
\begin{align}\nonumber
&\sum_{l=1}^{p}a_l\mu_{1,l}\Big[a_0^{(\gamma_l)}\nabla_tu_{ij}^n-
\sum_{k=1}^{n-1}(a_{n-k-1}^{(\gamma_l)}-a_{n-k}^{(\gamma_l)})\nabla_tu_{ij}^k-
a_{n-1}^{(\gamma_l)}\phi_{ij}\Big]\\\nonumber
&+b_1\nabla_tu_{ij}^n
+\sum_{m=1}^qc_m\mu_{2,m}\sum^n_{k=1}c_{n-k}^{(\alpha_m)}\nabla_tu_{ij}^k
+b_2u_{ij}^{n-\frac{1}{2}}\\\label{e27}
&=b_3(\delta_x^2+\delta_y^2)u_{ij}^{n-\frac{1}{2}}
+\sum_{r=1}^sd_r\mu_{3,r}\sum^n_{k=1}c_{n-k}^{(\beta_r)}\nabla_t(\delta_x^2+\delta_y^2)u_{ij}^k\\\nonumber
&+f_{ij}^{n-\frac{1}{2}}+R_{1ij}^n,
\end{align}
where $\mu_{1,l}=\frac{\tau^{1-\gamma_l}}{\Gamma(3-\gamma_l)}$,
 $\mu_{2,m}=\frac{\tau^{1-\alpha_m}}{\Gamma(2-\alpha_m)}$, $\mu_{3,r}=\frac{\tau^{1-\beta_r}}{\Gamma(2-\beta_r)}$, and
 $u(x_i,y_j,t_{n-\frac{1}{2}})=\frac{u(x_i,y_j,t_n)+u(x_i,y_j,t_{n-1})}{2}+O(\tau^2)$, $\frac{\partial}{\partial
 t}u(x_i,y_j,t_{n-\frac{1}{2}})=\frac{u(x_i,y_j,t_n)-u(x_i,y_j,t_{n-1})}{\tau}+O(\tau^2)$, so $|R_{1ij}^n|\leq
 C(\tau^{\min(3-\gamma_l,2-\alpha_m,2-\beta_r)}+h_x^2+h_y^2)$, in which $C$ is independent of $\tau$, $h_x$ and $h_y$. Omitting the error
 term, we use $U_{ij}^n$ as the numerical solution, then we obtain the implicit finite difference scheme for Eq.(\ref{e1})
\begin{align} \label{e27}
&\sum_{l=1}^{p}a_l\mu_{1,l}\Big[a_0^{(\gamma_l)}\nabla_tU_{ij}^{n}-
\sum_{k=1}^{n-1}(a_{n-k-1}^{(\gamma_l)}-a_{n-k}^{(\gamma_l)})\nabla_tU_{ij}^{k}-
a_{n-1}^{(\gamma_l)}\phi_{ij}\Big]\nonumber\\
&+b_1\nabla_tU_{ij}^{n}+\sum_{m=1}^qc_m\mu_{2,m}\sum^n_{k=1}c_{n-k}^{(\alpha_m)}\nabla_tU_{ij}^k
+b_2U_{ij}^{n-\frac{1}{2}}\nonumber\\
&=b_3(\delta_x^2+\delta_y^2)U_{ij}^{n-\frac{1}{2}}
+\sum_{r=1}^sd_r\mu_{3,r}\sum^n_{k=1}c_{n-k}^{(\beta_r)}\nabla_t(\delta_x^2+\delta_y^2)U_{ij}^k+f_{ij}^{n-\frac{1}{2}},
\end{align}

with initial and boundary conditions
\begin{align*}
U^0_{ij}&=\varphi_{ij},~(x_i,y_j)\in \Omega_h,\\
U_{ij}^n&=0,~(x_i,y_j)\in \mathcal{V}_h,~1\leq n\leq N.
\end{align*}

\section{Analysis of the numerical scheme}

\subsection{Solvability of the scheme}
Firstly, we discuss the solvability of the finite difference scheme (\ref{e27}).
\begin{theorem}\label{thm1}
The finite difference scheme (\ref{e27}) is uniquely solvable.
\end{theorem}
\begin{proof}
The finite difference scheme (\ref{e27}) can be recast into
\begin{align}
&\bigg(\frac{\sum\limits_{l=1}^{p}a_l\mu_{1,l}+b_1}{\tau}+\sum_{m=1}^qc_m\mu_{2,m}c_0^{(\alpha_m)}+\frac{b_2}{2}\bigg)U_{ij}^n
-\big(\frac{b_3}{2}+\sum_{r=1}^sd_r\mu_{3,r}c_0^{(\beta_r)}\big)(\delta_x^2+\delta_y^2)U_{ij}^n\nonumber\\
=&\big(\frac{\sum\limits_{l=1}^{p}a_l\mu_{1,l}+b_1}{\tau}-\frac{b_2}{2}\big)U_{ij}^{n-1}+\sum_{l=1}^{p}a_l\mu_{1,l}\Big[\sum_{k=1}^{n-1}(a_{n-k-1}^{(\gamma_l)}-a_{n-k}^{(\gamma_l)})\nabla_tU_{ij}^{k-\frac{1}{2}}+
a_{n-1}^{(\gamma_l)}\phi_{ij}\Big]\nonumber\\
+&\sum_{m=1}^qc_m\mu_{2,m}\Big[\sum_{k=1}^{n-1}(c_{n-k-1}^{(\alpha_m)}-c_{n-k}^{(\alpha_m)})U_{ij}^{k}+
c_{n-1}^{(\alpha_m)} U_{ij}^0\Big] \nonumber+\frac{b_3}{2}(\delta_x^2+\delta_y^2)U_{ij}^{n-1}\\\label{e28}
-&\sum_{r=1}^sd_r\mu_{3,r}\Big[\sum_{k=1}^{n-1}(c_{n-k-1}^{(\beta_r)}-c_{n-k}^{(\beta_r)})(\delta_x^2+\delta_y^2)U_{ij}^{k}+
c_{n-1}^{(\beta_r)}(\delta_x^2+\delta_y^2) U_{ij}^0\Big]+(f_{ij}^n+f_{ij}^{n-1})/2,
\end{align}
%rewrite Eq.(\ref{e29}) in the matrix form:
%\begin{eqnarray*}
% &{ AU^n=r^4U^{n-1}+r^5(\delta_x^2+\delta_y^2)U^{n-1}}+\frac{1}{2}(F^n+F^{n-1})\\[+1pt]
% &+\sum\limits_{l=1}^{p}a_l\mu_{1,l}\sum\limits_{k=1}^{n-1}(a_{n-k-1}^{(\gamma_l)}-a_{n-k}^{(\gamma_l)})\nabla_tU^{k-\frac{1}{2}}
% +\sum\limits_{l=1}^{p}a_l\mu_{1,l}a_{n-1}^{(\gamma_l)}\Phi\\[+1pt]
% &+\sum\limits_{m=1}^{q}c_m\mu_{2,m}\sum\limits_{k=1}^{n-1}(c_{n-k-1}^{(\alpha_m)}-c_{n-k}^{(\alpha_m)})U^k
% +\sum\limits_{m=1}^{q}c_m\mu_{2,m}c_{n-1}^{(\alpha_m)}\Psi\\[+1pt]
% &-\sum\limits_{j=1}^{s}d_j\mu_{3,j}\sum\limits_{k=1}^{n-1}(c_{n-k-1}^{(\beta_j)}-c_{n-k}^{(\beta_j)})(\delta_x^2+\delta_y^2)U^k
% -\sum\limits_{j=1}^{s}d_j\mu_{3,j}c_{n-1}^{(\beta_j)}(\delta_x^2+\delta_y^2)\Psi,
%\end{eqnarray*}
at each time level, the coefficient matrix $A$ is
\begin{equation*}
A=\left[\begin{array}{cccccc}
B & C & 0& \cdots & 0&0 \\
{C} & B &{C}& \cdots & 0 &0\\
0 & {C}&B & \cdots & 0 &0 \\
\vdots & \vdots & \vdots & \ddots &\vdots &\vdots\\
0 & 0 & 0 & \cdots &B &{C}\\
0 & 0& 0& \cdots &{C}& B \\
\end{array}\right],
\end{equation*}
where $B$ and $C$ are block matrixes,
\begin{equation*}
B=\left[\begin{array}{cccccc}
r_1+2r_2+2r_3 & -r_3 & 0& \cdots & 0&0 \\
-r_3 & r_1+2r_2+2r_3 &-r_3& \cdots & 0 &0\\
0 & -r_3&r_1+2r_2+2r_3& \cdots & 0 &0 \\
\vdots & \vdots & \vdots & \ddots &\vdots &\vdots\\
0 & 0 & 0 & \cdots &r_1+2r_2+2r_3 &-r_3\\
0 & 0& 0& \cdots &-r_3 & r_1+2r_2+2r_3 \\
\end{array}\right],
\end{equation*}
and
\begin{equation*}
C=\left[\begin{array}{cccccc}
-r_2 & 0& 0& \cdots & 0&0 \\
0 & -r_2 &0& \cdots & 0 &0\\
0 & 0&-r_2& \cdots & 0 &0 \\
\vdots & \vdots & \vdots & \ddots &\vdots &\vdots\\
0 & 0 & 0 & \cdots &-r_2 &0\\
0 & 0& 0& \cdots &0 & -r_2 \\
\end{array}\right],
\end{equation*}
where $
U^n=\left[u_{11}^n,u_{12}^n,\cdots,u_{1M_y-1}^n,u_{21}^n,u_{22}^n,\cdots,u_{2M_y-1}^n,\cdots,u_{M_x-11}^n,
u_{M_x-12}^n,\cdots,u_{M_x-1M_y-1}^n\right],\\
F^n=\left[f_{11}^n,f_{12}^n,\cdots,f_{1M_y-1}^n,f_{21}^n,f_{22}^n,\cdots,f_{2M_y-1}^n,\cdots,f_{M_x-11}^n,
f_{M_x-12}^n,\cdots,f_{M_x-1M_y-1}^n\right],\\
\Phi=\left[\phi_{11},\phi_{12},\cdots,\phi_{1M_y-1},\phi_{21},\phi_{22},\cdots,\phi_{2M_y-1},\cdots,
\phi_{M_x-11},\phi_{M_x-12},\cdots,\phi_{M_x-1M_y-1}\right],\\
\Psi=\left[\varphi_{11},\varphi_{12},\cdots,\varphi_{1M_y-1},\varphi_{21},\varphi_{22},\cdots,\varphi_{2M_y-1},
\cdots,\varphi_{M_x-11},\varphi_{M_x-12},\cdots,\varphi_{M_x-1M_y-1}\right],$
where $r_1=\frac{\sum_{l=1}^{p}a_l\mu_{1,l}+b_1}{\tau}+\sum_{m=1}^qc_m\mu_{2,m}c_0^{(\alpha_m)}+\frac{b_2}{2}>0$ ,
$r_2=\frac{\frac{b_3}{2}+\sum_{r=1}^sd_r\mu_{3,r}c_0^{(\beta_r)}}{h_x^2}>0$,
$r_3=\frac{\frac{b_3}{2}+\sum_{r=1}^sd_r\mu_{3,r}c_0^{(\beta_r)}}{h_y^2}>0$,
$r_4=\frac{\sum_{l=1}^{p}a_l\mu_{1,l}+b_1}{\tau}-\frac{b_2}{2}$ and $r_5=\frac{b_3}{2}$.
Then $A$ is a strictly diagonally dominant matrix. Therefore $A$ is nonsingular, which means that the numerical scheme (\ref{e27}) is uniquely
solvable.
\end{proof}
\subsection{Stability}
Then, we will analyze the stability of the schemes (\ref{e27}) by energy method.
\begin{theorem}\label{thm2}
The finite difference scheme (\ref{e27}) is unconditionally stable.
\end{theorem}
\begin{proof}
Multiplying Eq.(\ref{e17}) by $h_xh_y\tau \nabla_tU_{ij}^{n}$ and summing $i$ from 1 to $M_x-1$, $j$ from 1 to $M_y-1$ and summing $n$ from 1
to $N$, we obtain
\begin{align}
&\tau\sum_{n=1}^{N}\sum_{i=1}^{M_x-1}\sum_{j=1}^{M_y-1}\sum_{l=1}^{p}a_l\mu_{1,l}h_xh_y\Big[a_0^{(\gamma_l)}\nabla_tU_{ij}^{n}-
\sum_{k=1}^{n-1}(a_{n-k-1}^{(\gamma_l)}-a_{n-k}^{(\gamma_l)})\nabla_tU_{ij}^{k}-
a_{n-1}^{(\gamma_l)}\phi_{ij}\Big]\nabla_tU_{ij}^{n}
\nonumber\\
&+b_1\tau\sum_{n=1}^{N}\sum_{i=1}^{M_x-1}\sum_{j=1}^{M_y-1}h_xh_y(\nabla_tU_{ij}^{n})^2+
\tau\sum_{n=1}^{N}\sum_{i=1}^{M_x-1}\sum_{j=1}^{M_y-1}\sum_{m=1}^qc_m\mu_{2,m}h_xh_y\sum^n_{k=1}c_{n-k}^{(\alpha_m)}\nabla_tU_{ij}^k\nabla_tU_{ij}^{n}\nonumber\\\label{e29}
&+b_2\tau\sum_{n=1}^{N}\sum_{i=1}^{M_x-1}\sum_{j=1}^{M_y-1}h_xh_yU_{ij}^{n-\frac{1}{2}}\nabla_tU_{ij}^{n}
=b_3\tau\sum_{n=1}^{N}\sum_{i=1}^{M_x-1}\sum_{j=1}^{M_y-1}h_xh_y(\delta_x^2+\delta_y^2)U_{ij}^{n-\frac{1}{2}}\nabla_tU_{ij}^{n}\nonumber\\
&+\tau\sum_{n=1}^{N}\sum_{i=1}^{M_x-1}\sum_{j=1}^{M_y-1}\sum_{r=1}^sd_r\mu_{3,r}h_xh_y\sum^n_{k=1}c_{n-k}^{(\beta_r)}\nabla_t(\delta_x^2+\delta_y^2)U_{ij}^k\nabla_tU_{ij}^{n}\nonumber\\
&+\tau\sum_{n=1}^{N}\sum_{i=1}^{M_x-1}\sum_{j=1}^{M_y-1}h_xh_yf_{ij}^{n-\frac{1}{2}}\nabla_tU_{ij}^{n}.
\end{align}
For the first term, using Lemma \ref{lm6}, we have
\begin{align}
&\tau\sum_{n=1}^{N}\sum_{i=1}^{M_x-1}\sum_{j=1}^{M_y-1}\sum_{l=1}^{p}a_l\mu_{1,l}h_xh_y\Big[a_0^{(\gamma_l)}\nabla_tU_{ij}^{n}-
\sum_{k=1}^{n-1}(a_{n-k-1}^{(\gamma_l)}-a_{n-k}^{(\gamma_l)})\nabla_tU_{ij}^{k}-
a_{n-1}^{(\gamma_l)}\phi_{ij}\Big]\nabla_tU_{ij}^{n}\nonumber\\
\geq& \sum_{l=1}^{p}a_l\frac{\tau T^{1-\gamma_l}}{2\Gamma(2-\gamma_l)}\sum_{n=1}^{N}\sum_{i=1}^{M_x-1}\sum_{j=1}^{M_y-1}h_xh_y
(\nabla_tU_{ij}^{n})^2
-\sum_{l=1}^{p}a_l\frac{a_1 T^{2-\gamma_l}}{2\Gamma(3-\gamma_l)}\sum_{i=1}^{M_x-1}\sum_{j=1}^{M_y-1}h_xh_y\phi^2_{ij}\nonumber\\\label{e30}
=&\sum_{l=1}^{p}a_l\frac{\tau T^{1-\gamma_l}}{2\Gamma(2-\gamma_l)}\sum_{n=1}^{N}||\nabla_tU^{n}||_0^2
-\sum_{l=1}^{p}a_l\frac{T^{2-\gamma_l}}{2\Gamma(3-\gamma_l)}||\phi||_0^2.
\end{align}
For the second term, we have
\begin{align}\label{e31}
b_1\tau\sum_{n=1}^{N}\sum_{i=1}^{M_x-1}\sum_{j=1}^{M_y-1}h_xh_y(\nabla_tU_{ij}^{n})^2
=b_1\tau\sum_{n=1}^{N}||\nabla_tU^{n}||_0^2.
\end{align}
Using Lemma \ref{lm4}, we obtain
\begin{align}\label{e32}
&\tau\sum_{n=1}^{N}\sum_{i=1}^{M_x-1}\sum_{j=1}^{M_y-1}\sum_{m=1}^qc_m\mu_{2,m}h_xh_y\sum^n_{k=1}c_{n-k}^{(\alpha_m)}\nabla_tU_{ij}^k\nabla_tU_{ij}^{n}\nonumber\\
=& \sum_{m=1}^qc_m\mu_{2,m}\sum_{n=1}^{N}\sum_{k=1}^{n}c_{n-k}^{(\alpha_m)}(\nabla U^k,\nabla U^n)>0.
\end{align}
For the fourth term, we have
\begin{align}
&b_2\tau\sum_{n=1}^{N}\sum_{i=1}^{M_x-1}\sum_{j=1}^{M_y-1}h_xh_yU_{ij}^{n-\frac{1}{2}}\nabla_tU_{ij}^{n}=
\frac{b_2}{2}\sum_{n=1}^{N}(U^n+U^{n-1},U^n-U^{n-1})\nonumber\\\label{e33}
=&\frac{b_2}{2}\sum_{n=1}^{N}(||U^{n}||_0^2-||U^{n-1}||_0^2)=\frac{b_2}{2}(||U^{N}||_0^2-||U^{0}||_0^2).
\end{align}
Applying (\ref{e23}), we obtain
\begin{align}
&b_3\tau\sum_{n=1}^{N}\sum_{i=1}^{M_x-1}\sum_{j=1}^{M_y-1}h_xh_y(\delta_x^2+\delta_y^2)U_{ij}^{n-\frac{1}{2}}\nabla_tu_{ij}^{n}
=b_3\tau\sum_{n=1}^{N}((\delta_x^2+\delta_y^2)U^{n-\frac{1}{2}},\nabla_tU^{n})\nonumber\\
=&-\frac{b_3}{2}\sum_{n=1}^{N}
\langle(\nabla_x+\nabla_y)U^{n}+(\nabla_x+\nabla_y)U^{n-1},(\nabla_x+\nabla_y)U^n-(\nabla_x+\nabla_y)U^{n-1}\rangle\nonumber\\\label{e34}
=&-\frac{b_3}{2}\sum_{n=1}^{N}(|U^{n}|_1^2-|U^{n-1}|_1^2)
=\frac{b_3}{2}(|U^{0}|_1^2-|U^{N}|_1^2).
\end{align}
Combining (\ref{e14}) and Lemma \ref{lm4}, we have
\begin{align}\label{e35}
&\tau\sum_{n=1}^{N}\sum_{i=1}^{M_x-1}\sum_{j=1}^{M_y-1}\sum_{r=1}^sd_r\mu_{3,r}h_xh_y\sum^n_{k=1}c_{n-k}^{(\beta_r)}\nabla_t(\delta_x^2+\delta_y^2)U_{ij}^k\nabla_tU_{ij}^{n}\nonumber\\
=&\sum_{r=1}^sd_r\mu_{3,r}\tau\sum_{n=1}^{N}\sum^n_{k=1}c_{n-k}^{(\beta_r)}(\nabla_t((\delta_x^2+\delta_y^2)U^k),\nabla_tU^{n})\nonumber\\
=&-\sum_{r=1}^sd_r\mu_{3,r}\tau\sum_{n=1}^{N}\sum^n_{k=1}c_{n-k}^{(\beta_r)}
\langle\nabla_t((\nabla_x+\nabla_y)U^k),\nabla_t((\nabla_x+\nabla_y)U^n)\rangle<0.
\end{align}
For the last term and using the important inequality $ab\leq \varepsilon a^2+\frac{b^2}{4\varepsilon}$, we have
\begin{align}
&\tau\sum_{n=1}^{N}\sum_{i=1}^{M_x-1}\sum_{j=1}^{M_y-1}h_xh_yf_{ij}^{n-\frac{1}{2}}\nabla_tU_{ij}^{n}\nonumber\\
\leq& \tau\bigg(\sum_{l=1}^{p}\frac{a_l T^{1-\gamma_l}}{2\Gamma(2-\gamma_l)}+b_1\bigg)\sum_{n=1}^{N}\sum_{i=1}^{M_x-1}\sum_{j=1}^{M_y-1}h_xh_y
(\nabla_tU_{ij}^{n})^2\nonumber\\\label{e36}
&+\frac{\tau}{4\bigg(\sum_{l=1}^{p}\frac{a_lT^{1-\gamma_l}}{2\Gamma(2-\gamma_l)}+b_1\bigg)}
\sum_{n=1}^{N}\sum_{i=1}^{M_x-1}\sum_{j=1}^{M_y-1}h_xh_y(f_{ij}^{n-\frac{1}{2}})^2\\\nonumber
=&\tau\bigg(\sum_{l=1}^{p}a_l\frac{ T^{1-\gamma_l}}{2\Gamma(2-\gamma_l)}+b_1\bigg)\sum_{n=1}^{N}||\nabla_tU^{n}||_0^2+
\frac{\tau}{4\bigg(\sum_{l=1}^{p}\frac{a_lT^{1-\gamma_l}}{2\Gamma(2-\gamma_l)}+b_1\bigg)}
\sum_{n=1}^{N}||f^{n-\frac{1}{2}}||_0^2.
\end{align}
Substituting (\ref{e30})-(\ref{e36}) into (\ref{e27}), we have
\begin{align*}
&\tau\bigg(\sum_{l=1}^{p}\frac{a_l T^{1-\gamma_l}}{2\Gamma(2-\gamma_l)}+b_1\bigg)\sum_{n=1}^{N}||\nabla_tU^{n}||_0^2
-\sum_{l=1}^{p}\frac{a_l T^{2-\gamma_l}}{2\Gamma(3-\gamma_l)}||\phi||_0^2\\
&+\frac{b_2}{2}(||U^{N}||_0^2-||u^{0}||_0^2)\leq \frac{b_3}{2}(|U^{0}|_1^2-|U^{N}|_1^2)\\
+&\tau\bigg(\sum_{l=1}^{p}\frac{a_l T^{1-\gamma_l}}{2\Gamma(2-\gamma_l)}+b_1\bigg)\sum_{n=1}^{N}||\nabla_tU^{n}||_0^2+
\frac{\tau}{4\bigg(\sum_{l=1}^{p}\frac{a_lT^{1-\gamma_l}}{2\Gamma(2-\gamma_l)}+b_1\bigg)}
\sum_{n=1}^{N}||f^{n-\frac{1}{2}}||_0^2,
\end{align*}
then we have
\begin{align}
&{b_2}||U^{N}||_0^2+{b_3}|U^{N}|_1^2\nonumber\\
\leq& {b_2}||U^{0}||_0^2+{b_3}|U^{0}|_1^2\nonumber\\\label{e37}
+&\sum\limits_{l=1}^{p}\frac{a_l T^{2-\gamma_l}}{\Gamma(3-\gamma_l)}||\phi||_0^2+\frac{T}{\sum\limits_{l=1}^{p}\frac{a_l
T^{1-\gamma_l}}{\Gamma(2-\gamma_l)}+2b_1}
\max_{1\leq n\leq N}||f^{n-\frac{1}{2}}||_0^2.
\end{align}

From Eq.(23), the definition of $H^1$ norm, we have
\begin{align}
&||U^{N}||_{1}^2\leq |U^{0}|_{H^1}^2\nonumber\\\label{e38}
+&\sum\limits_{l=1}^{p}\frac{a_l T^{2-\gamma_l}}{\Gamma(3-\gamma_l)}||\phi||_0^2+\frac{T}{\sum\limits_{l=1}^{p}\frac{a_l
T^{1-\gamma_l}}{\Gamma(2-\gamma_l)}+2b_1}
\max_{1\leq n\leq N}||f^{n-\frac{1}{2}}||_0^2.
\end{align}
which means that the scheme (\ref{e27}) is unconditionally stable.
\end{proof}
\subsection{Convergence}
Now we discuss the convergence of the scheme (\ref{e27}).
\begin{theorem}\label{thm3}
Define ${u}^n$ and $U^n$
 as the exact solution and numerical solution vectors of scheme (\ref{e27}), respectively. Suppose that the solution to problem (\ref{e1})
 satisfies $u(x,y,t)\in C_{x,y,t}^{4,4,3}(\Omega)$, then there exists a positive constant $C$ independent of $h$ and $\tau$ such that
\begin{align*}
||{u}^n-U^n||_{H^1}&\leq C\sqrt{\frac{TL_xL_y }{{\sum_{l=1}^{p} \frac{a_l
T^{1-\gamma_l}}{\Gamma(2-\gamma_l)}}+2b_1}}(\tau^{\min(3-\gamma_l,2-\alpha_m,2-\beta_r)}+h_x^2+h_y^2).
\end{align*}
\end{theorem}
\begin{proof}
Denote $e_{ij}^n=u_{ij}^n-U_{ij}^n$, and $e^n$ is the error vector. Subtracting (\ref{e27}) from (\ref{e26}), we have
\begin{align*}
&\sum_{l=1}^{p}a_l\mu_{1,l}\Big[a_0^{(\gamma_r)}\nabla_te_{ij}^{n}-
\sum_{k=1}^{n-1}(a_{n-k-1}^{(\gamma_r)}-a_{n-k}^{(\gamma_r)})\nabla_te_{ij}^{k}\Big]+b_1\nabla_te_{ij}^{n}\\
+&\sum_{m=1}^{q}c_m\mu_{2,m}\Big[c_0^{(\alpha_m)}e_i^n-
\sum_{k=1}^{n-1}(c_{n-k-1}^{(\alpha_m)}-c_{n-k}^{(\alpha_m)})e_{ij}^k-
c_{n-1}^{(\alpha_m)} e_{ij}^0\Big]+b_2e_i^{n-\frac{1}{2}}\\
=&b_3(\delta_x^2+\delta_y^2)e_i^{n-\frac{1}{2}}+\sum_{r=1}^{s}d_r\mu_{3,r}\Big[c_0^{(\beta_r)}(\delta_x^2+\delta_y^2)e_{ij}^{n}-
\sum_{k=1}^{n-1}(c_{n-k-1}^{(\beta_r)}-c_{n-k}^{(\beta_r)})(\delta_x^2+\delta_y^2)e_{ij}^k\\
&-c_{n-1}^{(\beta_r)}(\delta_x^2+\delta_y^2) e_{ij}^0\Big]+R_{ij}^{n-\frac{1}{2}},
\end{align*}
with
\begin{align*}
e_{ij}^0&=0, \quad (x_i,y_j)\in\Omega_h,\\
e_{ij}^n&=0, \quad (x_i,y_j)\in\mathcal{V}_h,~1\leq n\leq N.
\end{align*}
Applying Theorem 2, and
\begin{align*}
|R_{ij}^{n-\frac{1}{2}}|\leq C(\tau^{\min(3-\gamma_l,2-\alpha_m,2-\beta_r)}+h_x^2+h_y^2).
\end{align*}
Then from (\ref{e38}), we can obtain
 \begin{align*}
||e^{n}||_{1}^2&\leq \frac{\tau h_xh_y}{{\sum\limits_{l=1}^p \frac{a_l T^{1-\gamma_l}}{\Gamma(2-\gamma_l)}}+2b_1}
\sum\limits_{k=1}^{n}\sum\limits_{i=1}^{M_x-1}\sum\limits_{i=1}^{M_y-1}(R_{ij}^k)^2\\
&\leq \frac{\tau h_xh_y}{{\sum\limits_{l=1}^p \frac{a_l T^{1-\gamma_l}}{\Gamma(2-\gamma_l)}}+2b_1}
\sum\limits_{k=1}^{n}\sum_{i=1}^{M_x-1}\sum\limits_{i=1}^{M_y-1}C^2(\tau^{\min(3-\gamma_l,2-\alpha_m,2-\beta_r)}+h_x^2+h_y^2)^2\\
&\leq \frac{C^2n\tau(M_x-1)h_x(M_y-1)h_y}{{\sum\limits_{l=1}^p \frac{a_l T^{1-\gamma_l}}{\Gamma(2-\gamma_l)}}+2b_1}
(\tau^{\min(3-\gamma_l,2-\alpha_m,2-\beta_r)}+h_x^2+h_y^2)^2\\
&\leq \frac{C^2TL_xL_y}{{\sum\limits_{l=1}^p\frac{a_l T^{1-\gamma_l}}{\Gamma(2-\gamma_l)}}+2b_1}
(\tau^{\min(3-\gamma_l,2-\alpha_m,2-\beta_r)}+h_x^2+h_y^2)^2.
\end{align*}
\end{proof}
This completes the proof of convergence of the difference scheme (27).
\section{Numerical examples}
\par In this section, we carry out some numerical experiments using the proposed finite difference schemes to illustrate our theoretical
statements.
\par \textbf{Example 1}\quad Consider the following multi-term time fractional diffusion equation
\begin{eqnarray*}
   \begin{cases}
   ~{D_t^{\gamma_1}} u(x,y,t)+{D_t^{\gamma_2}} u(x,y,t)+\frac{\partial u(x,y,t)}{\partial {t}}+{D_t^{\alpha_1}} u(x,y,t)+{D_t^{\alpha_2}}
   u(x,y,t)+u(x,y,t)\\[+1pt]
   ~=\frac{\partial^2 u(x,y,t)}{\partial {x^2}}+\frac{\partial^2 u(x,y,t)}{\partial {y^2}}
   +{D_t^{\beta_1}}(\frac{\partial^2 u(x,y,t)}{\partial {x^2}}+\frac{\partial^2 u(x,y,t)}{\partial {y^2}})\\[+1pt]
   ~+{D_t^{\beta_2}}(\frac{\partial^2 u(x,y,t)}{\partial {x^2}}+\frac{\partial^2 u(x,y,t)}{\partial {y^2}})+f(x,y,t), \\[+2pt]
   ~u(x,y,0)=\sin(\pi x)\sin(\pi y),\quad u_t(x,y,0)=0,\quad 0\le x\le 1,0\le y\le 1,\\[+1pt]
   ~u(0,y,t)=0,u(1,y,t)=0;\qquad u(x,0,t)=0,u(x,1,t)=0 \quad\quad 0\le t\le T,
\end{cases}
\end{eqnarray*}
where $(x,y,t)\in[0,1]\times[0,1]\times[0,T],$ $0<\alpha_1, \alpha_2, \beta_1, \beta_2<1$, $1<\gamma_1,\gamma_2<2$, the source term is
$ f(x,y,t)=\sin(\pi x)\sin(\pi
y)\Big[\frac{\Gamma(4)t^{3-\gamma_1}}{\Gamma(4-\gamma_1)}+\frac{\Gamma(4)t^{3-\gamma_2}}{\Gamma(4-\gamma_2)}+3t^2
+\frac{\Gamma(4)t^{3-\alpha_1}}{\Gamma(4-\alpha_1)}+\frac{\Gamma(4)t^{3-\alpha_2}}{\Gamma(4-\alpha_2)}+(1+2\pi^2)(t^3+1)+
\frac{2\pi^2\Gamma(4)t^{3-\delta_1}}{\Gamma(4-\delta_1)}+\frac{2\pi^2\Gamma(4)t^{3-\delta_2}}{\Gamma(4-\delta_2)}\Big],
$ and the exact solution is $u(x,y,t)=(t^3+1)\sin(\pi x)\sin(\pi y)$.
\par In this simulation, we choose $h_x=h_y=h,$ and use the implicit finite difference scheme (\ref{e27})
to solve the equation and the numerical results are given in Table 1 and Table 2.
Table 1 shows the $L_2$ error and $L_\infty$ error and the convergence order
of $h$ for different $\gamma_1$, $\gamma_2$, $\alpha_1$, $\alpha_2$, $\beta_1$, $\beta_2$ with $\tau=1/1000$ at $t=1$. Table 2 shows the $L_2$
error and $L_\infty$ error and the convergence order
of $\tau$ for different $\gamma_1$, $\gamma_2$, $\alpha_1$, $\alpha_2$, $\beta_1$, $\beta_2$ with
$\tau^{\min\{3-\gamma_1,3-\gamma_2,2-\alpha_1,2-\alpha_2,2-\beta_1,2-\beta_2\}}\approx h^2$ at $t=1$.
\par From the tables, we can find the numerical results are in good agreement with the exact solution and reach the accuracy of
$\tau^{\min\{3-\gamma_1,3-\gamma_2,2-\alpha_1,2-\alpha_2,2-\beta_1,2-\beta_2\}}+h^2$ order, which demonstrates the effectiveness of our
numerical method and confirms the theoretical analysis.

\begin{table}[htbp]% ================================== Table 1 =============================
\caption{The spacial error and convergence order for different $\gamma_1$, $\gamma_2$, $\alpha_1$, $\alpha_2$, $\beta_1$, $\beta_2$ with
$\tau=1/1000$.} \centering
\begin{tabular}{c c c c c }
 \toprule
$\gamma_1=1.8,~\gamma_2=1.6,~\alpha_1=\beta_1=0.8,~\alpha_2=\beta_2=0.6$ & $ L_2 $ error& Order & $L_\infty $ error& Order \\
\midrule
$h=1/4$ & 3.1425E-02 & & 6.2850E-02 & \\
$h=1/8$ & 7.7096E-03 & 2.03 & 1.5419E-02 & 2.03 \\
$h=1/16$ & 1.9092E-03 & 2.01 & 3.8185E-03 & 2.01 \\
$h=1/32$ & 4.6706E-04 & 2.03 & 9.3413E-04 & 2.03 \\
$h=1/64$ & 1.0701E-04 & 2.13 & 2.1402E-04 & 2.13 \\
\midrule
$\gamma_1=1.4,~\gamma_2=1.2,~\alpha_1=\beta_1=0.6,~\alpha_2=\beta_2=0.4$ & $ L_2 $ error & Order & $L_\infty$ error& Order \\
\midrule
$h=1/4$ & 3.1612E-02 & & 6.3224E-02 & \\
$h=1/8$ & 7.7561E-03 & 2.03 & 1.5512E-02 & 2.03 \\
$h=1/16$ & 1.9248E-03 & 2.01 & 3.8497E-03 & 2.01 \\
$h=1/32$ & 4.7517E-04 & 2.02 & 9.5033E-04 & 2.02 \\
$h=1/64$ & 1.1326E-04 & 2.07 & 2.2652E-04 & 2.07 \\
\bottomrule
\end{tabular}
\end{table}

\begin{table}[htbp]% ================================== Table 2 =============================
\caption{The temporal error and convergence order for different $\gamma_1$, $\gamma_2$, $\alpha_1$, $\alpha_2$, $\beta_1$, $\beta_2$ with
$\tau^{\min\{3-\gamma_1,3-\gamma_2,2-\alpha_1,2-\alpha_2,2-\beta_1,2-\beta_2\}}\approx h^2$.} \centering
\begin{tabular}{c c c c c }
 \toprule
$\gamma_1=1.8,~\gamma_2=1.6,~\alpha_1=\beta_1=0.8,~\alpha_2=\beta_2=0.6$ & $ L_2 $ error& Order & $L_\infty $ error& Order \\
\midrule
$\tau=1/20$ & 1.2344E-02 & & 2.4689E-02 & \\
$\tau=1/40$ & 5.4202E-03 & 1.19 & 1.0513E-02 & 1.23 \\
$\tau=1/80$ & 2.2146E-03 & 1.29 & 4.4293E-03 & 1.25 \\
$\tau=1/160$ & 9.8769E-04 & 1.16 & 1.9643E-03 & 1.17 \\
$\tau=1/320$ & 4.2609E-04 & 1.21 & 8.5217E-04 & 1.20 \\
\midrule
$\gamma_1=1.4,~\gamma_2=1.2,~\alpha_1=\beta_1=0.6,~\alpha_2=\beta_2=0.4$ & $ L_2 $ error & Order & $L_\infty$ error& Order \\
\midrule
$\tau=1/20$ & 6.1264E-03 & & 1.2253E-02 & \\
$\tau=1/40$ & 2.2731E-03 & 1.43 & 4.4801E-03 & 1.45 \\
$\tau=1/80$ & 8.6577E-04 & 1.39 & 1.7219E-03 & 1.38 \\
$\tau=1/160$ & 3.0586E-04 & 1.50 & 6.1050E-04 & 1.50 \\
$\tau=1/320$ & 1.1523E-04 & 1.41 & 2.3028E-04 & 1.41 \\
\bottomrule
\end{tabular}
\end{table}

\textbf{Example 2}\quad Consider the following model\cite{Fetecau}:
\begin{eqnarray*}
   \begin{cases}
   ~(1+\lambda^\alpha{D}_t^{\alpha})\frac{\partial u(x,y,t)}{\partial t}=\nu(1+\lambda^\beta{D}_t^\beta)(\frac{\partial^2}{\partial
   x^2}+\frac{\partial^2}{\partial y^2})u(x,y,t),\\[+1pt]
   ~u(x,y,0)=u_t(x,y,0)=0,\quad x>0, \quad0\le y\le d,\\[+1pt]
   ~ u(x,0,t)=u(x,d,t)=0,\quad x>0, \quad t\geq 0, \\[+2pt]
   ~ u(0,y,t)=At,u(L,y,t)=0,\quad x>0,\quad t\geq0,
\end{cases}
\end{eqnarray*}
where $d$ is the distance between the two side walls, $L$ is the distance of the plate in $x$-direction.
In the calculation, we choose $h=1/20$, $\tau=1/100$, $L=d=5$.
In order to observe the effects of different physical parameter on
the velocity field, we plot some figures to demonstrate the dynamic
characteristics of the generalized Oldroyd-B fluid. The variations of $u(x,y,t)$ with $x,y$ for different values of $\lambda,\nu,\alpha,\beta$
at a fixed time($t=1$) are illustrated in Figs. 1-2, from the figures, we can conclude that the parameter $\lambda,\nu$ and the fractional
order $\alpha, \beta$ have effects on the velocity function $u(x,y,t)$. Fig 3 shows the influence of time on the velocity and we can note that
the flow velocity increases with $t=0.5$ and $t=1$ respectively.
\begin{figure}[!h]
\begin{center}
\begin{minipage}{0.47\textwidth}\centering
\epsfig{figure=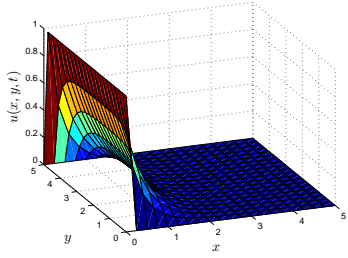,width=5cm} \par {(a) $\lambda=4,~\nu=1$}
\end{minipage}
%\begin{minipage}{0.47\textwidth}\centering
%\epsfig{figure=eg3case2,width=6cm}\par {(b) Case II, $N=24$, $a_1=b_1=a_2=b_2=0$.}
%\end{minipage}\\
\begin{minipage}{0.47\textwidth}\centering
\epsfig{figure=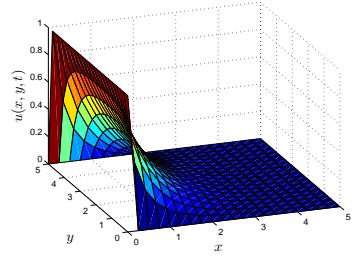,width=5cm} \par {(b) $\lambda=1,~\nu=2$}
\end{minipage}
%\begin{minipage}{0.47\textwidth}\centering
%\epsfig{figure=eg3m6,width=6cm} \par {(d) $M=6,\mu=2/3,a_1=b_2=-1/2,a_2=b_1=1/2$.}
%\end{minipage}
\end{center}
\caption{Numerical solution profiles of velocity $u(x,y,t)$ at $~\alpha=0.8,~\beta=0.6.$ }
\end{figure}
\begin{figure}[!h]
\begin{center}
\begin{minipage}{0.47\textwidth}\centering
\epsfig{figure=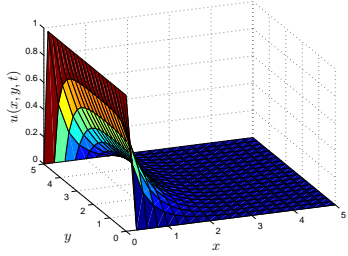,width=5cm} \par {(a) $~\alpha=0.6,~\beta=0.5$ }
\end{minipage}
%\begin{minipage}{0.47\textwidth}\centering
%\epsfig{figure=eg3case2,width=6cm}\par {(b) Case II, $N=24$, $a_1=b_1=a_2=b_2=0$.}
%\end{minipage}\\
\begin{minipage}{0.47\textwidth}\centering
\epsfig{figure=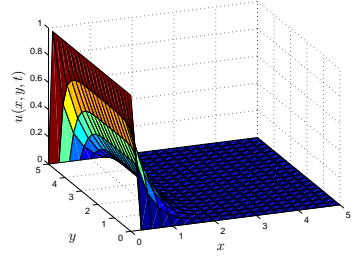,width=5cm} \par {(b) $~\alpha=0.8,~\beta=0.2$}
\end{minipage}
%\begin{minipage}{0.47\textwidth}\centering
%\epsfig{figure=eg3m6,width=6cm} \par {(d) $M=6,\mu=2/3,a_1=b_2=-1/2,a_2=b_1=1/2$.}
%\end{minipage}
\end{center}
\caption{Numerical solution profiles of velocity $u(x,y,t)$ at $~\lambda=2,~\nu=1$.}
\end{figure}

\begin{figure}[!h]
\begin{center}
\begin{minipage}{0.47\textwidth}\centering
\epsfig{figure=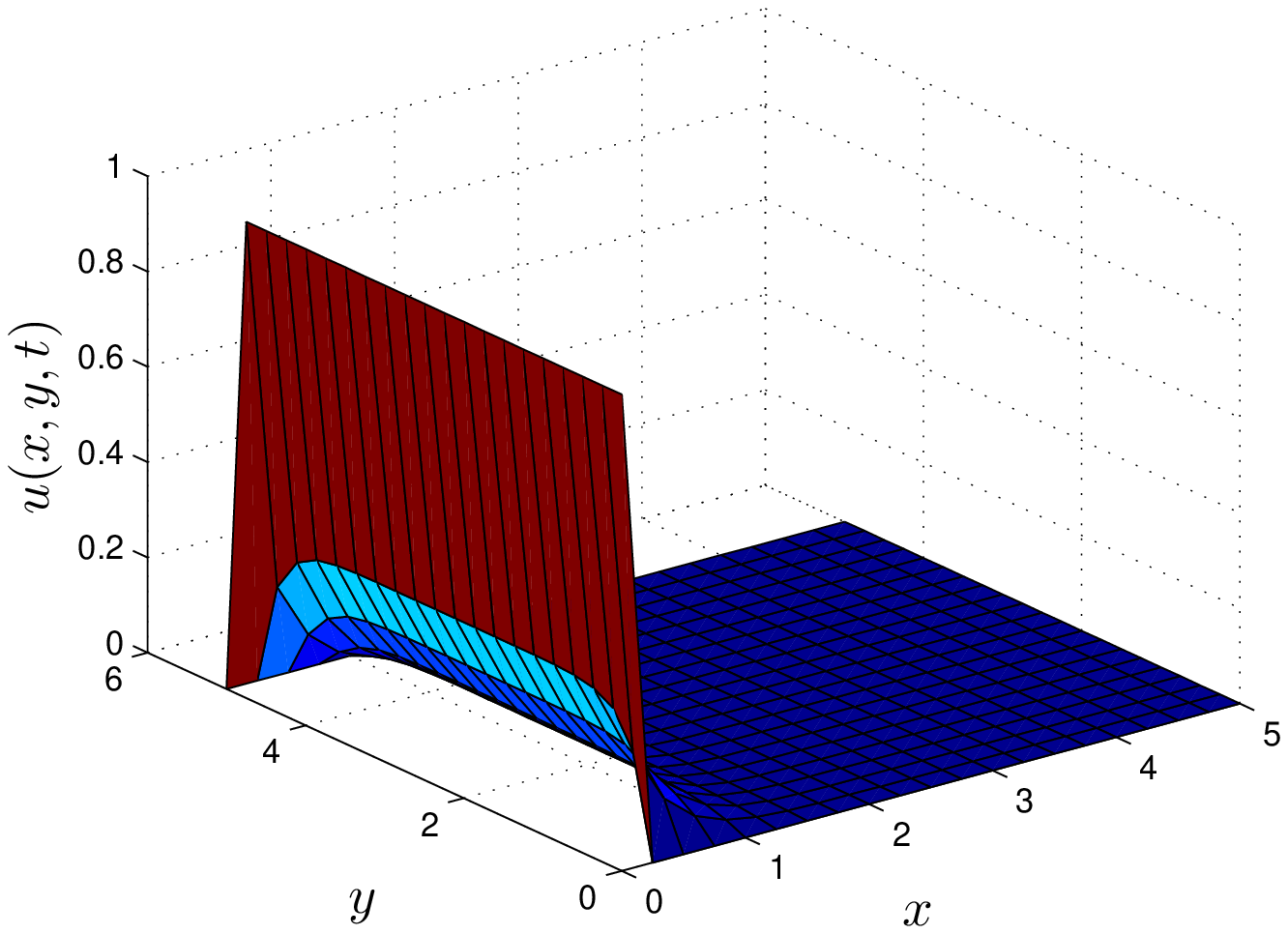,width=5cm} \par {(a) $~t=0.5$}
\end{minipage}
%\begin{minipage}{0.47\textwidth}\centering
%\epsfig{figure=eg3case2,width=6cm}\par {(b) Case II, $N=24$, $a_1=b_1=a_2=b_2=0$.}
%\end{minipage}\\
\begin{minipage}{0.47\textwidth}\centering
\epsfig{figure=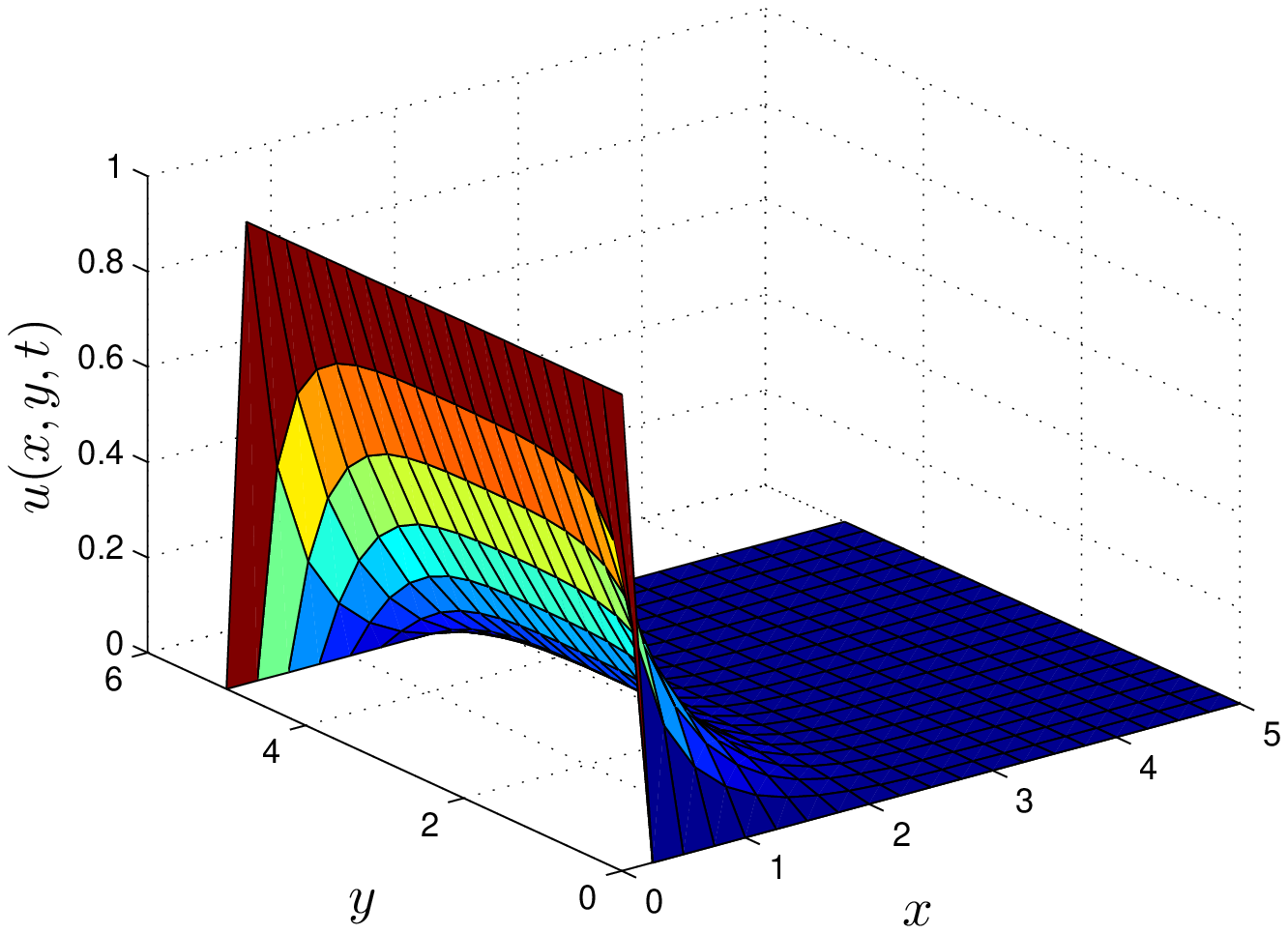,width=5cm} \par {(b) $~t=1$}
\end{minipage}
%\begin{minipage}{0.47\textwidth}\centering
%\epsfig{figure=eg3m6,width=6cm} \par {(d) $M=6,\mu=2/3,a_1=b_2=-1/2,a_2=b_1=1/2$.}
%\end{minipage}
\end{center}
\caption{Numerical solution profiles of velocity $u(x,y,t)$ of at $~\alpha=0.8,~\beta=0.4,~\lambda=5,~\nu=2$.}
\end{figure}

\section{Conclusion}
In this paper, we proposed a finite difference method to solve the multi-term time fractional diffusion equation incorporating the unsteady
MHD Couette flow of a generalized Oldroyd-B fluid. We give a implicit finite difference schemes with accuracy of
$O(\tau^{\min\{3-\gamma_l,2-\alpha_m,2-\beta_r\}}+h_x^2+h_y^2)$. In addition, we established the stability and convergence analysis for the
implicit difference scheme. Two numerical examples were exhibited to verify the effectiveness and reliability of our method. We can conclude
that our numerical method is robust and can be extended to other multi-term time fractional diffusion equations, such as the generalized
Oldroyd-B fluid in a rotating system and the generalized Maxwell fluid model. In future work, we shall investigate alternating direction
implicit (ADI) method to the two-dimensional generalized Oldroyd-B fluid, convert the two-dimensional
computation to several one-dimensional ones, and reduce the computing time and storage.\\
\noindent {\bf Acknowledgements}
\par The first author wishes to acknowledge that this work was partially supported by National Natural Science Foundation of China (Nos.11801060).
The second two author wish to acknowledge that this research was partially supported by the Australian Research Council (ARC) via the Discovery Project (DP180103858).

%% For one-column wide figures use
%\begin{figure}
%% Use the relevant command to insert your figure file.
%% For example, with the graphicx package use
% \includegraphics{example.eps}
%\label{fig:1} % Give a unique label
%\end{figure}
%%
%% For two-column wide figures use
%\begin{figure*}
%% Use the relevant command to insert your figure file.
%% For example, with the graphicx package use
% \includegraphics[width=0.75\textwidth]{example.eps}
% \includegraphics[width=0.75\textwidth]{example.eps}
%\label{fig:2} % Give a unique label
%\end{figure*}
%

%\begin{acknowledgements}
%\end{acknowledgements}

%\bibliographystyle{spbasic} % basic style, author-year citations
%\bibliographystyle{spmpsci} % mathematics and physical sciences
%\bibliographystyle{spphys} % APS-like style for physics
%\bibliography{} % name your BibTeX data base


\end{document}